%% file: dubouloz_lamy_final.tex
\documentclass[11pt]{amsart}

\usepackage{amssymb}
\usepackage[cmtip,all]{xy}
\usepackage{pslatex} 
\usepackage{enumerate}
\usepackage{pstricks,pst-node}
\advance\hoffset by -20mm
\newtheorem{thm}{Main Theorem}
\newtheorem{theorem}{Theorem}
\newtheorem{prop}[thm]{Proposition}
\newtheorem{lem}[thm]{Lemma}
\newtheorem{cor}[thm]{Corollary}

\theoremstyle{definition}
\newtheorem{definition}[thm]{Definition}

\theoremstyle{remark}
\newtheorem{rem}[thm]{Remark}

\numberwithin{equation}{section}

\newcommand{\Z}{\mathbb{Z}}
\newcommand{\C}{\mathbb{C}}

\newcommand{\F}{\mathbb{F}}
\newcommand{\CP}{\mathbb{P}}
\newcommand{\cpo}{\mathbb{P}^1}
\newcommand{\cpd}{\mathbb{P}^2}
\newcommand{\cpt}{\mathbb{P}^3}

\newcommand{\pt}{\text{pt}}
\newcommand{\reste}{R}

\newcommand{\ind}{{\rm ind}}
\newcommand{\bs}{\mathcal{B}}
\newcommand{\jun}{\mathrm{TA}}
\newcommand{\GTA}{\mathrm{GTA}}
\newcommand{\M}{\mathcal{M}}

\DeclareMathOperator{\tr}{Tr}
\DeclareMathOperator{\Aut}{Aut}

\input{graph_surface}

\usepackage[colorlinks, linktocpage, citecolor = blue, linkcolor = blue]{hyperref}

\begin{document}

\title{Automorphisms of open surfaces with irreducible boundary}

\author{Adrien Dubouloz}
\address{Institut de Math\'ematiques de Bourgogne, Universit\'e de Bourgogne, 9 avenue Alain Savary - BP 47870, 21078 Dijon cedex, France}
\email{adrien.dubouloz@u-bourgogne.fr}

\author{St\'ephane Lamy}
\address{Institut de Math\'ematiques de Toulouse,
	     Universit\'e Paul Sabatier, 	
	     118 route de Narbonne,
	     31062 Toulouse Cedex 9, France}
\email{slamy@math.univ-toulouse.fr}
\thanks{The second author was partially supported by a Marie Curie Intra European Fellowship to visit the Mathematics Institute of Warwick, on leave from the Institut Camille Jordan, Universit\'e Lyon 1.}
\thanks{Both authors were partially supported by ANR Grant ``BirPol''  ANR-11-JS01-004-01.}
\date{March 2014}

\begin{abstract}
Let $(S,B_S)$ be the log pair associated with a projective completion of a  smooth quasi-projective surface $V$. 
Under the assumption that the boundary $B_S$ is irreducible, we obtain an algorithm to factorize any automorphism of $V$ into a sequence of simple links.
This factorization lies in the framework of the log Mori theory, with the property that all the blow-ups and contractions involved in the process occur on the boundary. 
When the completion $S$ is smooth, we obtain a description of the automorphisms of $V$ which is reminiscent of a presentation by generators and relations except that the ``generators" are no longer automorphisms. 
They are instead isomorphisms between different models of $V$ preserving certain rational fibrations. 
This description enables one to define normal forms of automorphisms and leads in particular to a natural generalization of the usual notions of affine and Jonqui\`eres automorphisms of the affine plane. 
When $V$ is affine, we show however that except for a finite family of surfaces including the affine plane, the group generated by these affine and Jonqui\`eres automorphisms, which we call the tame group of $V$, is a proper subgroup of $\Aut(V)$.

\end{abstract}
%

\maketitle

\section*{Introduction}

Smooth affine surfaces with a rich group of algebraic automorphisms have been intensively studied after the pioneering work of M.H. Gizatullin and V.I. Danilov in the seventies. 
Affine surfaces whose automorphism group acts with a dense orbit with finite complement were first characterized by M.H. Gizatullin \cite{G} in terms of the structure of their boundary divisors in smooth minimal projective completions.
Namely, except for finitely many exceptional cases, these surfaces are precisely those which admit completions by chains of proper nonsingular rational curves. 
Their automorphism groups have been studied by V.I. Danilov and M.H. Gizatullin in a series of papers \cite{GD1,GD2}. 
They established in particular that their automorphism groups can be described as fundamental groups of  graphs of groups attached to well-chosen families of projective completions. 
The vertices of these graphs correspond to classes up to isomorphism of suitable projective models of the affine surfaces under consideration while the arrows are determined by certain birational relations between these. 
It is however difficult to extract from them more concrete geometric properties of automorphisms or the existence of interesting subgroups due to the fact that they have in general uncountably many vertices and uncountably many edges between any pairs of vertices.
  
Affine surfaces $V$ as above have the nice geometric property that they come equipped with families of $\mathbb{A}^{1}$-fibrations $\pi \colon V\rightarrow\mathbb{A}^{1}$, that is, surjective morphisms with general fibers isomorphic to the affine line. 
The original approach of M.H. Gizatullin and V.I. Danilov has been recently reworked by J. Blanc and the first author \cite{BD} with a particular focus on the interactions between automorphisms and these fibrations. 
This led to introduce simpler graphs encoding  equivalence classes of rational fibrations from which it is possible to decide for instance if the automorphism group of $V$ is generated by automorphisms preserving these fibrations. 
However the methods used in \emph{loc. cit.} remain close to the ones introduced by M.H. Gizatullin and V.I. Danilov, depending in particular on properties of birational maps that are a priori specific to the $2$-dimensional case.

As a step towards a hypothetical theory to study automorphisms of higher dimensional affine varieties by methods of birational geometry, it is natural to try to reformulate these existing results in the framework of log Mori theory. 
Since every smooth affine surface admitting a completion by a chain of smooth rational curves admits in fact such a completion by a particular chain $C_{0},C_{1},\ldots,C_{r}$, $r\geq1$, whose self-intersections are respectively $0,a_{1},\ldots,a_{r}$, where $a_{1}\le-1$ and $a_{i}\le-2$ for all $i=2,\dots,r$, we see that after contracting the curves $C_{1},\ldots,C_{r}$, we obtain a completion by a possibly singular projective surface $S$ with an irreducible boundary $B_{S}=C_{0}$. 
So given a smooth quasi-projective surface $V$, we would like more generally to describe the automorphism group of $V$ say when there exists a completion $S  \supset V$ where $S$ is a possibly singular projective surface with $S \smallsetminus V$ equal to an irreducible curve.  
More precisely, we look for a factorization in the framework of the log Mori theory for automorphisms of $V$ that do not extend as biregular automorphisms on $S$. 

When  $V$ admits a completion into a (log) Mori fiber space $S$, and $f\colon S\dashrightarrow S$  is the birational self-map  associated to an automorphism of $V$, the (log) Sarkisov program gives a factorization of $f$ into so-called elementary links between intermediate (log) Mori fiber spaces. 
As already expressed in \cite{BM}, the hope is that a  refinement of such an algorithm could allow to understand the structure of polynomial automorphisms of the affine $3$-space $\mathbb{A}^3$. 
Here we have in mind to complete  $\mathbb{A}^3$ by the projective space $\cpt$, and to apply the algorithm to the birational map from $\cpt$ to $\cpt$ induced by an automorphism of $\mathbb{A}^3$. 
It seems  natural to expect an algorithm which is \emph{proper}, that is where all the blow-ups and contractions occur on the boundary divisor. 

A natural first step is to check if at least in the $2$-dimensional case, the log Sarkisov program satisfies this property, and so could  be used to give a good description of the automorphism groups of quasi-projective surfaces $V$ admitting completions into log Mori fiber spaces.  
But maybe surprisingly it turns out that applying the log Sarkisov program to such a completion $S$ does not provide a satisfactory description: 
In general the links occurring in a factorization of a birational transformation of $S$ induced by an automorphism of $V$ do not preserve the inner quasi-projective surface $V$ (see Proposition \ref{pro:notproper}). 
This is not the case for $\mathbb{A}^2$, but it is worth noting that the phenomenon occurs for the 3-dimensional affine space:  
There exist some automorphisms of $\mathbb{A}^3$ for which the usual Sarkisov factorization is \emph{not} proper (see \cite[\S 1.2.3]{HDR}). \\

This motivated the search for an alternative algorithm for which all the blow-ups and contractions would occur on the boundary divisor.
It is such an algorithm, together with applications and examples, that we propose in this paper, the main point being a shift in focus from the existence of completions with a log Mori fiber space structure to the existence of completions by one irreducible divisor. This last property might turn out to be the right one for studying automorphisms of $\mathbb{A}^3$. 

Before stating our main result, let us introduce the class of \emph{dlt completions} of a smooth quasi-projective surface $V$: 
These are divisorially log terminal pairs $(S,B_S)$ consisting of a projective completion  $S$ of $V$ and a reduced boundary divisor $B_S= \sum E_i$, such that the support of $B_S$ is exactly $S\smallsetminus V$. 
Also, by a \emph{strictly birational map of dlt completions} we mean a birational map $f:(S,B_S)\dashrightarrow (S',B_{S'})$ which induces an isomorphism $S\smallsetminus B_S\rightarrow S'\smallsetminus B_{S'}$ and which is not a biregular isomorphism. With these definitions, our factorization result reads as follows. 

\begin{theorem}\label{th:main}
Let $f\colon V\stackrel{\sim}{\rightarrow}V'$ be an isomorphism of smooth quasi-projective surfaces,  and let $S,S'$ be dlt completions of $V$ and $V'$ with irreducible boundary divisors  $B_S, B_{S'}$. 
Then if the induced map  $f\colon  S \dashrightarrow S'$ is strictly birational, we can factorize $f$ into a finite sequence of $n$ links of the following form
 $$
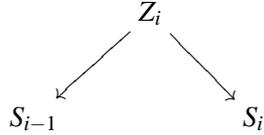
\xymatrix{
 &Z_i \ar[dl] \ar[dr] \\
 S_{i-1} && S_i
 }$$
where  $S_0 = S,S_1, \dots,S_n = S'$ are dlt completions of  $V$ with an irreducible boundary,  $Z_i$ is for all $i = 1,\dots,n$ a dlt completion of $V$ with two boundary components,  and $Z_i \to S_{i-1}$, $Z_i \to S_i$ are the divisorial contractions associated with each one of the two $K+B$ extremal rays with support in the boundary of ${Z_i}$. 
\end{theorem}

The existence of the above decomposition, which was already considered by the authors in \cite{DL} (unpublished), is in fact a  particular case of more general factorization results developed later on by Y.M. Polyakova: see \cite{Pol} where she reformulates the problem  in terms of relations induced by certain classes of birational maps in suitable categories of $2$-dimensional log-terminal pairs. 
This approach certainly provides a nice theoretical framework for studying automorphisms of quasi-projective surfaces in general:  for instance, one can recover from it the description of M.H Gizatullin and V.I. Danilov in terms of fundamental groups of graphs of groups. 
However, it remains too abstract to give precise handle on the properties of these automorphism groups and their subgroups. 
In our view, such a factorization result is only a preliminary step for the understanding of these groups, and a second crucial step consists in  extracting from it some particular classes of birational maps which are relevant for the study of precise properties of these groups. 
For instance, in \cite{BD} the question was to decide whether the automorphism group of an affine surface admitting a completion by a chain of smooth rational curves is generated by automorphisms preserving $\mathbb{A}^1$-fibrations.
The problem was solved by introducing two classes of birational maps called \emph{fibered modifications} and \emph{reversions}, roughly characterized by the respective properties that they preserve an $\mathbb{A}^1$-fibration or exchange it to another one, and then by using an appropriate factorization result to deduce that any automorphism can be decomposed in a finite sequence of such maps. \\

Here, as an application of our factorization result, we follow a similar strategy to describe the structure of the automorphism group of a quasi-projective surface $V$ admitting a smooth completion $(S,B_S)$ with irreducible boundary $B_S\simeq \mathbb{P}^1$, a case which is essentially complementary to the situations in which the combinatorial methods developed in \cite{BD} give a satisfactory description.
Affine surfaces of this type have been first studied by Gizatullin and Danilov \cite{GD2}: 
They established in particular that their isomorphy types as abstract affine surfaces depend only on the self-intersection $B_S^2$ of the boundary $B_S$ in a smooth completion $(S,B_S)$ and not on the choice of a particular smooth completion $S$ or boundary divisor $B_S$ (except in the case $B_S^2 = 4$ where there are two models). 
They described their automorphism groups in terms of the action of certain groups on a ``space of tails'' which essentially encodes the isomorphy types of smooth completions $(S,B_S)$ of a fixed affine surface $V$. 
Here we follow a different approach based on a natural generalizations of the classical notions of \emph{Affine} and \emph{Jonqui\`eres} automorphisms for the affine plane. Roughly, for a given affine surface $V$, affine automorphisms in our sense are characterized by the property that they come as restrictions of biregular automorphisms of various smooth completions $(S,B_S)$ while Jonqui\`eres automorphisms are automorphisms which preserve certain $\mathbb{A}^1$-fibration on $V$. 
With these notions, we obtain a kind of presentation by generators and relations closely related to the one considered by Gizatullin and Danilov in \emph{loc. cit.} and reminiscent of the usual description given by Jung's Theorem for automorphisms of the affine plane. 

It is classical that $\Aut(V)$ is generated by these two classes of automorphisms when $V$ is $\mathbb{A}^2$ or a smooth affine quadric surface and we are able to prove that this holds more generally for every affine surface $V$ admitting a smooth completion $(S,B_S)$ with rational irreducible boundary of self-intersection $B_S^2\leq 4$.
On the other hand, we show that this property fails for those admitting smooth completions $(S,B_S)$ with $B_S^2\geq 5$.
We also derive from our description that if $B_S^2\geq 5$ then $\Aut(V)$ is ``much bigger" than the automorphism group of $\mathbb{A}^2$,  in the sense that the proper normal subgroup  generated by affine and Jonqui\`eres automorphisms of $V$ cannot be generated by a countable family of algebraic subgroups (see Proposition \ref{prop:Affautos}). \\

The article is organized as follows. In Section \ref{sec:2},  we briefly review the log Sarkisov program and we illustrate the reason why it does not provide a satisfactory algorithm to obtain informations about automorphism groups of quasi-projective surfaces.
In Section \ref{sec:1}, we review the geometry of dlt completions, establish our factorization Theorem \ref{th:main} and discuss some of its properties. 
Then in Section \ref{sec:3} we apply our algorithm to the case of quasi-projective surfaces $V$ admitting smooth completions with irreducible boundaries. 
We observe that our algorithm yields a kind of presentation by generators and relations for the automorphisms of $V$ (Proposition \ref{prop:main}) and enables to define a notion of normal forms for automorphisms. 
We then consider the situation where $V$ is affine and discuss the structure of the automorphism group (Proposition \ref{prop:Affautos}). 
Finally, section \ref{sec:4} is devoted to the explicit study of various examples of affine surfaces $V$ admitting smooth completions with irreducible boundaries which illustrate the increasing complexity of the groups $\Aut(V)$ in terms of the self-intersection of their boundary divisors.

\section{Quasi-projective surfaces with log Mori fiber space completions and
the log Sarkisov program}

\label{sec:2} Many interesting quasi-projective surfaces with a rich
automorphism group admit completions into dlt pairs $(S,B_{S})$ which
are log Mori fiber spaces $g\colon(S,B_{S})\to Y$, i.e., $g$ has
connected fibers, $Y$ is a normal curve or a point, and all the curves
contracted by $g$ are numerically proportional and of negative intersection
with the divisor $K_{S}+B_{S}$. Examples of such situations include
the affine plane $\mathbb{A}^{2}$ or quasi-projective surfaces 
obtained as complements of either a section or a fiber in a $\mathbb{P}^1$-bundle 
over a smooth projective curve. In this context, the log Sarkisov program established by
Bruno and Matsuki \cite{BM} gives an effective algorithm to factorize
a birational map $f\colon S\dashrightarrow S'$ between log Mori fiber
spaces into a sequence of elementary links for which we control the
complexity of the intermediate varieties in the sense that at any
step they differ from a log Mori fiber space by at most one divisorial contraction.
As it was established by Takahashi \cite[p.401]{Mat} for the case of $\mathbb{A}^2$, it
seems natural to expect in general that given a quasi-projective surface $V$ and a log Mori fiber space $S$ completing $V$, 
applying this algorithm to birational maps $f\colon S\dashrightarrow S$
corresponding to automorphisms of $V$ would lead to a good description
of the automorphism group of $V$. 
Unfortunately, this is not the case as it turns out in general that
the birational transformations involved in the algorithm do not preserve the inner quasi-projective surface
$V$. In this section we briefly review the mechanism of the log Sarkisov
program of Bruno and Matsuki and illustrate this phenomenon.

\subsection{Overview of the log Sarkisov program for projective surfaces}

\label{subsec:11} 

Let $f\colon S\dashrightarrow S'$ be a birational
map between 2-dimensional log Mori fiber spaces $(S,B_{S})$ and $(S',B_{S'})$.
We assume further that the latter are log MMP related, i.e. that they can be both obtained 
from a same pair $(X,B_{X})$ consisting of a smooth surface $X$ and a simple normal crossing divisor $B_{X}$
by running the log Minimal Model Program. We denote by $\pi \colon X\rightarrow S$ the corresponding morphism and by $C_{i}\subset X$ the irreducible components of its exceptional locus. 

The algorithm depends on two main discrete invariants of the birational map $f$ which are defined as follows. 
First, we choose an ample divisor $H'$ on $S'$. 
We denote by $H_{S}\subset S$ (resp. $H_{X}\subset X$, etc...) the strict transform of a general member
of the linear system $|H'|$. The \emph{degree} $\mu$ of $f$ is then defined as the positive
rational number $\frac{H_{S}.C}{-(K_{S}+B_{S}).C}$ where $C$ is
any curve contained in a fiber of the log Mori fiber structure on $S$. 
For the second invariant, the fact that $\pi$ is obtained by running the log MMP implies that in the ramification formulas
\[K_{X}+B_{X}=\pi^{*}(K_{S}+B_{S})+\sum a_{i}C_{i},\quad H_{X}=\pi^{*}H_{S}-\sum m_{i}C_{i}\]
we have $a_i>0$ for every $i$, which enables to define the \emph{maximal multiplicity} $\lambda$ of $f$ as the maximum of the positive rational numbers $\lambda_{i}=\frac{m_{i}}{a_i}$. 

If $\lambda>\mu$, then the algorithm  predicts the existence of
a maximal extraction, that is, an extremal divisorial contraction $Z\to S$ whose exceptional divisor realizes the maximal
multiplicity $\lambda$. Then either $Z$ is itself a log Mori fiber
space, or there exists another extremal divisorial contraction from 
$Z$ that brings us back to a log Mori fiber space. These operations
done, one shows that we have simplified $f$ in the sense that: either
$\mu$ went down; or $\mu$ remained constant but $\lambda$ went
down; or $\mu$ and $\lambda$ remained constant but the number of
exceptional divisors in $X$ realizing the maximal multiplicity $\lambda$
went down. 
Otherwise, if $\lambda\le\mu$, the algorithm predicts that either $S$
is equipped with a second structure of log Mori fiber space for which
the associated degree $\mu$ is strictly smaller, or there exists an
extremal divisorial contraction from $S$ to another log Mori fiber
space for which $\mu$ is again strictly smaller. 

The four types of elementary links occurring in the factorization procedure can be summarized
by the following diagrams:

\begin{figure}[ht]
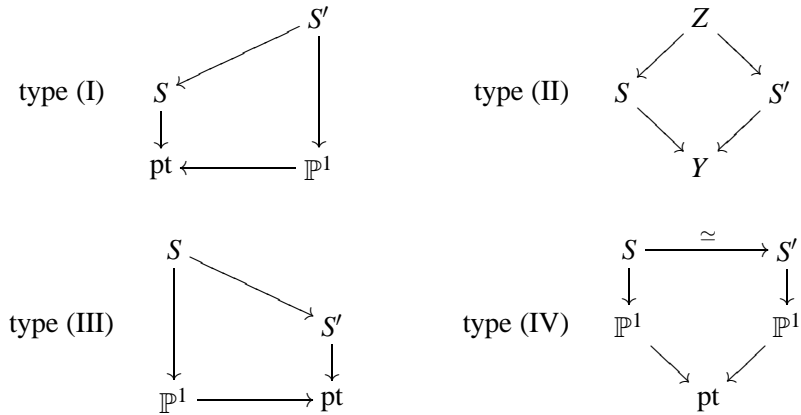

\begin{tabular}{ccc}
 type (I) $\quad \xygraph{
!{<0cm,0cm>;<0.7cm,0cm>:<0cm,1cm>::}
!{(0,0)}*+{S}="S" !{(3,1)}*+{S'}="S'" 
!{(0,-1)}*+{\pt}="pt" !{(3,-1)}*+{\cpo}="cpo"
"S"-@{->}"pt"  "S'"-@{->}"cpo"  
"cpo"-@{->}"pt" "S'"-@{->}"S"
}$ & $\qquad$ &
 type (II) $\quad\xygraph{
!{<0cm,0cm>;<0.7cm,0cm>:<0cm,1cm>::}
!{(0,0)}*+{S}="S" !{(3,0)}*+{S'}="S'" 
!{(1.5,1)}*+{Z}="Z" !{(1.5,-1)}*+{Y}="C" 
"S"-@{->}"C"  "S'"-@{->}"C"  
"Z"-@{->}"S" "Z"-@{->}"S'" 
}$\\
\\
type (III) $\quad\xygraph{
!{<0cm,0cm>;<0.7cm,0cm>:<0cm,1cm>::}
!{(0,1)}*+{S}="S" !{(3,0)}*+{S'}="S'"  
!{(0,-1)}*+{\cpo}="cpo" !{(3,-1)}*+{\pt}="pt"
"S"-@{->}"cpo"  "S'"-@{->}"pt"  
"cpo"-@{->}"pt" "S"-@{->}"S'"
}$ &&
type (IV) $\quad\xygraph{
!{<0cm,0cm>;<0.7cm,0cm>:<0cm,1cm>::}
!{(0,1)}*+{S}="S" !{(3,1)}*+{S'}="S'" 
!{(0,0)}*+{\cpo}="cpo" !{(3,0)}*+{\cpo}="cpo'"
!{(1.5,-1)}*+{\pt}="pt"
"S"-@{->}"cpo"  "S'"-@{->}"cpo'"  
"S"-@{->}^{\simeq}"S'" "cpo"-@{->}"pt" "cpo'"-@{->}"pt"
}$\\
\end{tabular} 
\caption{The four types of links of the log Sarkisov program.} \label{fig:LSP}
\end{figure}

The above program works for 2-dimensional dlt pairs $(S,B_{S})$. 
Bruno and Matsuki \cite{BM} also established the existence
of the analogue program in dimension $3$ for Kawamata log terminal
(klt) pairs $(Y,B_{Y})$ generalizing the original $3$-dimensional
version previously written down by Corti \cite{Cor}. For klt pairs
in any dimension, Hacon and McKernan \cite{HMc} recently gave a proof
of the existence of a factorization of birational maps between log
Mori fiber spaces into sequences of links of types (I), $\dots$,
(IV) (the definition of these links is slightly more
complicated in higher dimension because of the presence of isomorphisms
in codimension 1). However, their description, based on the results
in \cite{BCHM}, is much less effective and does not take the form of an explicit algorithm. 
In any case, we shall see in the next subsection that anyone of these factorization results
is in general inadequate to study the automorphism group of an open
surface $V$.

\subsection{Inadequacy of the log Sarkisov program}
The following criterion shows that for a large class of quasi-projective
surfaces $V$ admitting completions into log Mori fiber spaces $(S,B_{S})$,
any procedure which factors a birational map $S\dashrightarrow S$
into sequences of links of types (I), $\dots$, (IV) between log
Mori fiber spaces will affect in a nontrivial way the inner surface
$V$.

\begin{prop} \label{pro:notproper} Let $V$ be a quasi-projective
surface admitting a completion into a log Mori fiber space $\rho\colon S\rightarrow C$ over a smooth projective curve $C$. 
Suppose further that each irreducible component of the boundary $S\smallsetminus V$  has nonnegative self-intersection, and is not contained in a fiber of any log Mori fiber space structure on $S$. 
Then a strictly birational map $\phi\colon S\dashrightarrow S$,
cannot admit a factorization into a sequence of Sarkisov links
of type (I), $\dots$, (IV), each restricting to an isomorphism on $V$. \end{prop}

\begin{proof} Since $\rho\colon S\rightarrow C$ is a log Mori fiber
space over a curve, an elementary link starting from $S$ is necessarily
of type (II), (III) or (IV). Links of type (IV) only change the considered
log Mori fiber space structure on $S$ to another structure of the
same type. Since $\phi$ is strictly birational, it cannot be factored
into a sequence of links of type (IV). Therefore, after a sequence
of links of type (IV), one has necessarily to perform a link of type
(II) or (III) with respect to the log Mori fiber space structure $\rho'\colon S\rightarrow C$
at this step. Since by assumption the components of the boundary have
non-negative self-intersection hence cannot be contracted, we see that
a link of type (III) never restricts to an isomorphism on $V$. Consider
now the possibility of a link of type (II). After performing the extraction
$Z\rightarrow S$ with center at a point $q\in S$, the morphism $Z\to S'$
is the contraction of the strict transform of the unique fiber $F$
of the log Mori fiber space $\rho':S\rightarrow C$ passing though
$q$. Our hypothesis implies that $F$ is not an irreducible component
of the boundary $S\smallsetminus V$, and so, the link does not restrict
to an isomorphism on $V$. \end{proof}

\label{par:inadequacy}

\subsection{Example} As an illustration of Proposition \ref{pro:notproper}, let us consider
the case of the smooth affine surface $V$ defined as the complement of the diagonal $D$ in $\mathbb{P}^{1}\times\mathbb{P}^{1}$. 
The birational map 
\[
f\colon(x,y)\in\mathbb{A}^{2}\dashrightarrow\left(x+\frac{1}{x-y},y+\frac{1}{x-y}\right)\in\mathbb{A}^{2}
\]
preserves the levels $x-y=\mbox{\it constant}$, and extends via the embedding $(x,y)\in\mathbb{A}^{2}\hookrightarrow([x:1],[y:1])\in\mathbb{P}^{1}\times\mathbb{P}^{1}$ to a birational map from $S=\mathbb{P}^{1}\times\mathbb{P}^{1}$ to
$S'=\mathbb{P}^{1}\times\mathbb{P}^{1}$ inducing an isomorphism on
$V=\mathbb{P}^{1}\times\mathbb{P}^{1}\smallsetminus D$, where $D$ is the closure
of the diagonal $x-y=0$ in $\mathbb{A}^2$. The unique proper%
\footnote{ By \textit{proper} we mean a base point which is not an infinitely
near point.%
} base point of $f$ is the point $p=([1:0],[1:0])$, and the unique
contracted curve is the diagonal $D$. Straightforward calculations in local charts show that we
can resolve $f$ by performing 4 blow-ups that give rise to divisors $C_{1},\dots,C_{4}$ arranged
as on Figure \ref{fig:1}. We denote by $C_{0}$ the strict transform
of the diagonal $D$. Note that $C_{4}$ is the strict transform of
the diagonal in $S'$.

Choosing $H'=D$ as an ample divisor on $S'$, the coefficients $a_i$ in the ramification formulas
\[
K_{X}+B_{X}=\pi^{*}(K_{S}+B_{S})+\sum a_{i}C_{i},\;\mbox{ and }\; H_{X}=\pi^{*}H_{S}-\sum m_{i}C_{i},
\]
are easy to compute. For the $m_{i}$, one exploits for instance the fact that the strict transform
$H_{S}$ of a general member of $|D|$ is a rational curve of bidegree
$(3,3)$ with a double point at $p$ and at each of the infinitely near base points of $f$.
The results are tabulated in Figure \ref{fig:1}.

\begin{figure}[ht]
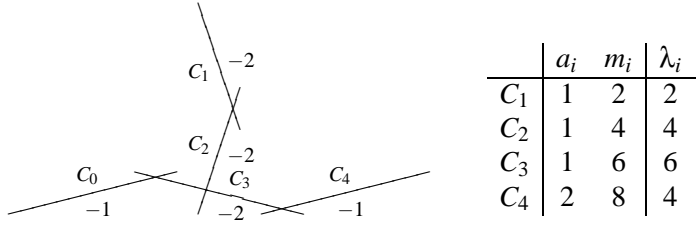

\[
\dessinresolutionbasic\qquad\begin{array}{c|cc|c}
 & a_{i} & m_{i} & \lambda_{i}\\
\hline C_{1} & 1 & 2 & 2\\
C_{2} & 1 & 4 & 4\\
C_{3} & 1 & 6 & 6\\
C_{4} & 2 & 8 & 4
\end{array}
\]
 \caption{Resolution of $f$ and coefficients in the ramification formulas.}

\label{fig:1} 
\end{figure}
The maximal multiplicity is thus realized by the divisor $C_{3}$ and a maximal
extraction $Z\rightarrow S$ is obtained by first blowing-up three times to produce $C_{1},C_{2}$
and $C_{3}$ and then contracting $C_{1}$ and $C_{2}$ creating a cyclic
quotient singularity. The boundary $Z\smallsetminus V$ consists of two irreducible curves $C_{0}$
and $C_{3}$, the latter supporting the unique singular point on the
surface. Furthermore, there exist 4 irreducible curves on $Z$ that correspond to $K+B$ extremal rays: 
\begin{itemize}
\item $C_{3}$, which is the exceptional divisor associated with the maximal
multiplicity; 
\item $C_{0}$, which is the strict transform of the diagonal on $S$; 
\item The strict transforms of the 2 rules $D_{+}$ and $D_{-}$ of $\mathbb{P}^{1}\times\mathbb{P}^{1}$
crossing at $p$. 
\end{itemize}
Now the log Sarkisov program imposes to contract one of the two curves
$D_{+}$ or $D_{-}$ above (precisely: the one that was a fiber for
the chosen structure of log Mori fiber space on $\mathbb{P}^{1}\times\mathbb{P}^{1}$)
to reach a new log Mori fiber space. But this birational contraction
does not restrict to an isomorphism on the affine surface $V$. 

However, the above computation shows that we are left with a third option which consists
in contracting the strict transform $C_{0}$ of $D$. This is precisely
the curve that our alternative algorithm will impose to contract to
get a new projective surface $S_{1}$ supporting a cyclic quotient singularity 
along the new boundary $B_{S_{1}}=C_{3}$. 
By construction, the corresponding birational map $S\dashrightarrow S_{1}$
induces an isomorphism on the inner affine surface $V$ but it turns out that $S_1$ is no longer a log Mori fiber space. 
Indeed, its divisor class group is isomorphic to $\mathbb{Z}^{2}$, generated by the strict transforms of $D_{+}$
and $D_{-}$. 
On the other hand, one checks that these curves generate the only $K+B$ extremal rays on $S_{1}$, each of these giving rise to a divisorial contraction $S_1\rightarrow \mathbb{P}^2$. 
Note in particular that even though it consists of a maximal extraction $Z\rightarrow S$ followed by a divisorial contraction $Z\rightarrow S_{1}$, the birational map $S\dashrightarrow S_{1}$ just constructed is not a Sarkisov link
of type (II). \\

Summing up, Proposition \ref{pro:notproper} and the above example
show that for quasi-projective surfaces $V$ admitting completions
into log Mori fiber spaces, there does not exist any factorization
process for which each elementary step is simultaneously a link of type (I), $\dots$, (IV) 
between log Mori fiber spaces and a birational map restricting to an isomorphism on $V$. 
So we cannot escape the dilemma that 
inevitably we have to abandon
one of these properties.

\section{The factorization algorithm}

\label{sec:1}

Here we first review basic facts on 2-dimensional dlt pairs and discuss the geometry of the boundaries of dlt completions involved in our main statement.
Then we prove  Main Theorem \ref{th:main} and discuss some additional properties of the factorization.

\subsection{Singularities and geometry of boundaries} 
The fact that an automorphism of a normal quasi-projective surface
$V$ extends to an automorphism of its minimal desingularisation enables to restrict without 
loss of generality to the case of a smooth quasi-projective surface. 
On the other hand since an extremal contraction starting from a smooth
log surface may yield a singular one, it is necessary to allow some kind
of singularities on the projective completions $S$ of $V$. Following recent work of Fujino \cite{Fuj10}, 
the widest framework where the log Mori Program is established in dimension $2$ is the one of pairs $(S,B_{S})$
with log canonical singularities. However, it is enough for our purpose
to work with the subclass of dlt pairs $(S,B_{S})$. 

\subsubsection{Hirzebruch-Jung singularities} \label{sec:sing} Before giving the characterization of these pairs that will be used in the sequel, let us first recall that an isolated
singular point $p$ of a surface $S$ is called a Hirzebruch-Jung cyclic quotient singularity of type
$A_{n,q}$, $n\geq 2$, $1\leq q \leq n-1$, $\gcd(n,q)=1$ if it is analytic locally isomorphic to the quotient of $\mathbb{A}^{2}$
by the action of the group $\mu_n\simeq \mathbb{Z}/n\mathbb{Z}$ of complex $n$-th roots of unity
defined by $(x,y)\mapsto (\varepsilon x,\varepsilon^q y)$. As it is well-known (see e.g. \cite[page 99]{BHPV}), the exceptional locus of the minimal resolution $\pi\colon\overline{S}\to S$ of $p$ consists of a chain of rational curves $E_{1},\dots,E_{s}$ with self-intersections $E_{i}^{2}=-a_{i}\leq -2$ determined by the expression 
\[
\frac{n}{q}=a_{1}-{\displaystyle \frac{1}{a_{2}-\frac{1}{a_{3}-\dots}}}
\]
as a continued fraction. Recall that cyclic quotient singularities are log terminal, i.e., in the ramification formula $K_{\overline{S}}=\pi^{*}K_{S}+\sum c_{i}E_{i}$ one has $-1<c_i$ for every $i$. 
For such singularities, one has in fact $-1<c_i\leq 0$. Indeed, otherwise, we can write $K_{\overline{S}}=\pi^{*}K_{S}+A-B$ where $A$ and $B$ are effective $\mathbb{Q}$-divisors supported on the exceptional locus of $\pi$ and without common components.
Since $A^2<0$, it follows that $K_{\overline{S}}\cdot A=(A-B)\cdot A<0$ and hence, there would exist an index $i$ such that $K_{\overline{S}}\cdot E_i<0$. 
But then $E_i$ would be a $(-1)$-curve which is absurd.

\subsubsection{Dlt pairs} For a definition of such pairs in general, we refer the reader to \cite[Definition 2.8]{KS}.
In our situation, \cite[Proposition 2.42]{KM} combined with the local
description of log terminal singularities of surfaces which can be
found in \cite[see in particular page 57, case (3)]{Kflips} leads to the 
following equivalent definition:

\begin{definition} A pair $(S,B_S)$ consisting of a projective surface $S$ and a nonempty reduced divisor $B_{S}=\sum B{}_{i}$ 
such that $S\smallsetminus B_{S}$ is smooth is called divisorially log terminal (dlt) if the following conditions are satisfied:
\begin{itemize}
\item The $B_{i}$ are smooth irreducible curves with normal crossings, that is each common point of two components is a normal crossing at a smooth point of $S$; 
\item A singular point $p$ of $S$ is a Hirzebruch-Jung
singularity $A_{n,q}$ and the strict transform of $B_S$ in the minimal resolution $\pi\colon\bar{S}\to S$ of $p$ meets the exceptional chain of rational curves $E_{1},\dots,E_{s}$ transversally at a unique point of the initial or final curve $E_{1}$ or $E_s$.
\end{itemize}

\noindent In particular, a dlt pair $(S,B_S)$ with irreducible boundary divisor $B_S$ is a purely log terminal (plt) pair.  
\end{definition}

Note that the above conditions guarantee in particular that the total transform of $B_S$ in the minimal resolution $\pi\colon\tilde{S}\rightarrow S$ of the singularities of $S$ is a simple normal crossing divisor.

\subsubsection{Geometry of the boundary} \label{sec:boundary} Let us first introduce notations and terminology that will be used in the sequel.
Given a strictly birational map of dlt completions $f\colon (S,B_S)\dashrightarrow (S',B_{S'})$ with irreducible boundaries, we denote by $\pi\colon\tilde{S}\rightarrow S$ and $\pi'\colon\tilde{S}'\rightarrow S'$ the minimal resolutions of the singularities of $S$ and $S'$ respectively. 
We denote by $\tilde{S}\stackrel{\sigma}{\leftarrow}X\stackrel{\sigma'}{\rightarrow}\tilde{S}'$ the minimal resolution of the base points of the birational map $\tilde{f}\colon\tilde{S}\dashrightarrow\tilde{S}'$ induced by $f$. 
Recall \cite[Theorem 1.3.7]{AC}
that $X$ and the birational morphisms $\sigma,\,\sigma'$ are uniquely
determined up to isomorphism by the following universal property:
given another resolution $\tilde{S}\leftarrow X'\rightarrow\tilde{S}'$,
there exists a unique birational morphism $X'\to X$ such that the
obvious diagram commutes. In particular, $X$ does not contain $(-1)$-curves that are exceptional for both $\pi \circ \sigma:X\to {S}$ and $\pi' \circ \sigma':X\to {S}'$.
This implies that if the sequence of blow-ups $\sigma':X\to\tilde{S}'$ is not empty, the $(-1)$-curve produced as the last exceptional divisor of the sequence is the strict transform of $B_S$.
Note also that by construction the boundary of $\tilde S$ and $X$ are simple normal crossing divisors, with each irreducible component a smooth rational curve.   

\noindent The following result shows that the existence of strictly birational maps of dlt completions  $f:(S,B_S)\dashrightarrow (S',B_{S'})$ imposes strong constraints on the boundaries:
\begin{prop} \label{prop:admissible} Let $f\colon (S,B_S)\dashrightarrow (S',B_{S'})$ 
be a strictly birational map of dlt completions with irreducible boundaries.
Then the following holds:
\begin{enumerate}
\item The boundaries $B_{S}$ and $B_{S'}$ are both isomorphic to $\mathbb{P}^{1}$; 
\item $S$ admits at most two singularities;
\item $f$ admits a unique proper base point $\bs(f)$, and if $S$ has exactly two singularities then $\bs(f)$ coincides with one of these singularities. 
\end{enumerate}
\end{prop} 

\begin{proof}Recall (see e.g. \cite[Theorem 5.2 page 410]{Har})
that if $h:M\dashrightarrow M'$ is a birational map between normal
surfaces, and $p\in M$ is a proper base point of $h$, then there
exists a curve $C\subset M'$ such that $h^{-1}(C)=p$. In our situation, 
since $B_{S'}$ is the only curve that can be transformed to a point by $f^{-1}$, it follows that
$f$ has a unique proper base point $\bs(f)=f^{-1}(B_{S'})\in B_S$. This implies
in turn that $f(B_{S})$ cannot be equal to $B_{S'}$ and so must
be equal to a point $p'=\bs(f^{-1})\in B_{S'}$. In particular, with the notation above, the strict transforms on the minimal resolution $X$ of $B_{S}$ and $B_{S'}$ are smooth rational curves (they come either from the resolution of a $A_{n,q}$ singularity, or from the blow-up of a smooth point), and they are not equal. 
This gives (1).

Now suppose that the union of the singularities of $S$ and of $\bs(f)$
consists of at least three distinct points supported
on $B_{S}$. 
Then on $X$, the strict transform of $B_{S}$ is a boundary component with at least three neighbors.
If $\sigma'\neq\text{id}$ then the first contraction must be
the one of the strict transform of $B_S$, which is impossible since the boundary divisor is simply normal crossing for all surfaces between $X$ and $\tilde{S}'$. 
Hence $\sigma'=\text{id}$,
but again this gives a contradiction, since on $\tilde{S}'$ all divisors
except maybe the strict transform of $B_{S'}$ which is distinct from that of $B_S$ must have at most two neighbors.
This proves (2) and (3). \end{proof}

\subsection{Proof of the factorization Theorem \ref{th:main}}

\label{section:preuve} The proof relies on the following lemma which
characterizes the possible extremal rays supported on the boundaries
of dlt completions $(S,B_{S})$.

\begin{lem} \label{lem:KBneg} Let $(S,B_{S})$ be a dlt completion of a smooth quasi-projective surface $V$. 
\begin{enumerate}
\item A (smooth) rational curve $C\subset B_{S}$ with at least two neighboring
components in $B_{S}$ is not a $K_{S}+B_{S}$ extremal ray. 
\item If $C\subset B_{S}$ is a smooth rational curve with only one neighboring
component in $B_{S}$ and supporting at most one singularity of $S$,
then $(K_{S}+B_{S}).C<0$. 
\item Let $C\subset B_{S}$ be a curve supporting exactly one singularity
$p$ of $S$, and denote by $\overline{C}$ the strict transform of
$C$ in the minimal resolution of $p$. If ${\overline{C}}^{2}<0$
then $C^{2}<0$. 
\end{enumerate}
\end{lem}

\begin{proof} Let $n$ be the number of neighbors of $C$ in $B_{S}$
and let $p_{1},\ldots,p_{r}$ the singular points of $S$ supported
along $C$. By the adjunction formula (see e.g. 2.2.4 in \cite{Pro}),
we have 
\[
(K_{S}+B_{S})\cdot C=(K_{S}+C)\cdot C+n=\deg(K_{C}+\text{Diff}_{C}(0))+n=-2+\sum_{i=1}^{r}(1-\frac{1}{m_{i}})+n,
\]
 where $m_{i}\geq2$ is the index of the singular point $p_{i}$,
$i=1,\ldots,r$. This implies (1) and (2). For (3), let $\pi:\overline{S}\rightarrow S$ be a minimal resolution
of $p$ and let $E=E_1$ be the unique $\pi$-exceptional curve that intersects
the strict transform $\overline{C}$ of $C$. 
We write $\overline{C}=\pi^{*}C-bE-\reste$, $K_{\overline{S}}=\pi^*(K_S)+cE+\reste'$
where $b>0$, $0 \ge c > -1$ (see  \S\ref{sec:sing}) and where $\reste,\reste'$
are $\pi$-exceptional divisors whose supports do not meet $\overline{C}$.
The fact that $(S,B_{S})$ is a dlt pair implies that 
$c-b>-1$ whence that $1>b$. The assertion follows since $C^{2}=\pi^{*}C\cdot\overline{C}=(\overline{C}+bE+\reste)\cdot\overline{C}={\overline{C}}^{2}+b<{\overline{C}}^{2}+1$.
\end{proof}

\begin{proof}[Proof of Theorem \ref{th:main}] Recall that we
have a strictly birational map $f\colon (S,B_S)\dashrightarrow (S',B_{S'})$ 
restricting to an isomorphism $V=S\smallsetminus B_{S}\simeq V'=S'\smallsetminus B_{S'}$.
As in subsection \ref{sec:boundary}, we let $\pi\colon\tilde{S}\rightarrow S$
and $\pi'\colon\tilde{S}'\rightarrow S'$ be the minimal resolutions
of singularities and we let $\tilde{S}\stackrel{\sigma}{\leftarrow}X\stackrel{\sigma'}{\rightarrow}\tilde{S}'$
be the minimal resolution of the base points of the induced birational map $\tilde{f}$. 
By Proposition \ref{prop:admissible}(1) and the description of Hirzebruch-Jung singularities given in \S\ref{sec:sing}, the divisor $B_{X}$ is then a tree of rational curves.
The irreducible components of $B_{X}$
are exceptional for at least one of the two morphisms $\pi\circ\sigma$
or $\pi'\circ\sigma'$, thus they all have a strictly negative self-intersection.
Since $B_{X}$ is a tree, there exists a unique sub-chain $E_{0},E_{1},\ldots,E_{n}={E_0}'$
of $B_{X}$ joining the strict transforms $E_{0}$ and ${E_0}'$ of $B_{S}$ and $B_{S'}$ respectively. We proceed by induction on the number $n+1$ of components
in this chain. The integer $n\geq 1$  will also be the number of links needed to factorize $f$. 
We use the same notation for the curves $E_{i}$, $i=0,\ldots,n$ and
their images or strict transforms in the different surfaces that will
come into play. 

To construct the first link $S=S_0\dashrightarrow S_1$, we consider the minimal partial resolution $\tilde{S} \leftarrow Y\dashrightarrow \tilde{S}'$ of $\tilde{f}$ dominated by $X$ and containing the divisor $E_1$ defined as follows:

\indent - If $\tilde{f}:\tilde{S}\dashrightarrow\tilde{S}'$
is either a morphism or has a proper base point supported outside 
from $E_{0}$, then $E_{1}$ is one of the exceptional divisor of
$\pi$, and the boundary $B_{\tilde{S}}$ is a chain of rational curves
with $E_{0},E_{1}$ intersecting in one point. In this
case we put $Y=\tilde{S}$.

\indent - Otherwise, if $\tilde{f}:\tilde{S}\dashrightarrow\tilde{S}'$
has a proper base point on $E_0$ then by definition of the resolution $X$,
the divisor $E_{1}$ is produced by blowing-up successively the base
points of $\tilde{f}$ as long as they lie on $E_{0}$, $E_{1}$ being the
last divisor produced by this process. We let $Y\rightarrow\tilde{S}$ 
be the intermediate surface thus obtained. By construction, the image
of the curves contracted by the induced birational morphism $X\to Y$
are all located outside $E_{0}$ and the self-intersections of $E_{0}$
in $X$ and $Y$ are equal. The divisor $B_{Y}$ is a chain that looks as in Figure \ref{fig:proof}. 
\begin{figure}[ht]
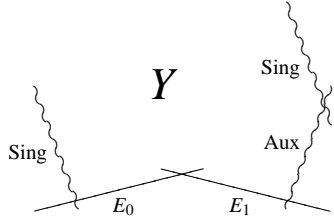

 $$\dessinpreuve$$
 \caption{The boundary divisor of $Y$.} \label{fig:proof}
\end{figure}
The wavy curves labeled ``Sing'' correspond to the (possible) chains
of rational curves obtained by desingularisation of $S$, and the
wavy curve labeled ``Aux'' corresponds to the (possible) chain
of auxiliary rational curves, each with self-intersection $-2$,
obtained by resolving the base points of $\tilde f$ before getting $E_{1}$.

In both cases,  we have $E_{0}^{2}<0$ on $Y$,
since this self-intersection is the same as the one on $X$.
So all irreducible components of $B_Y$ have a strictly negative self-intersection.   By running the $K+B$ MMP on $Y$ we can successively
contract all the components of the boundary $B_{Y}$ except $E_{0}$ and $E_{1}$. 
Indeed at each step $Y$ is a minimal resolution of the intermediate surface, and each extremity component
$C$ of the boundary chain supports at most one singularity: Lemma
\ref{lem:KBneg} ensures that $C$ is $K+B$ negative and has negative
self-intersection whence generates a $K+B$ extremal
ray giving rise to a divisorial extremal contraction. We note
$(Z,E_{0}+E_{1})$ the dlt pair obtained from the pair $(Y,B_{Y})$
by this sequence of contractions. 
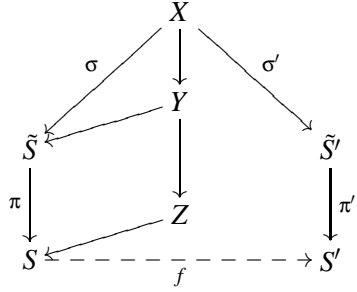
\[
\mygraph{!{<0cm,0cm>;<1cm,0cm>:<0cm,1cm>::}!{(2,3.3)}*+{X}="X"!{(2,2.1)}*+{Y}="Y"!{(2,0.6)}*+{Z}="Z"!{(0,0)}*+{S}="S"!{(4,0)}*+{S'}="S'"!{(0,1.5)}*+{\tilde{S}}="Stilda"!{(4,1.5)}*+{\tilde{S'}}="S'tilda""X"-@{->}_{\sigma}"Stilda""X"-@{->}^{\sigma'}"S'tilda""X"-@{->}"Y""Y"-@{->}"Stilda""Y"-@{->}"Z""Z"-@{->}"S""Stilda"-@{->}_{\pi}"S""S'tilda"-@{->}^{\pi'}"S'""S"-@{-->}_{f}"S'"}
\]
By construction, $Z$ dominates $S$ via the divisorial contraction
of the $K+B$ extremal curve $E_{1}$. Again by Lemma \ref{lem:KBneg},
$E_{0}$ generates a $K+B$ extremal ray in $Z$, and $E_{0}^{2}<0$
on $Z$. So there exists a $K+B$ divisorial extremal contraction
$Z\rightarrow S_{1}$ contracting $E_{0}$ and yielding a new dlt
pair $(S_{1},B_{S_{1}})$ with reduced boundary $B_{S_{1}}$ consisting
of the strict transform of $E_{1}$. We obtain the first expected link and the
map $f\colon S\dashrightarrow S'$ factorizes via a birational map
$f_{1}\colon S_{1}\dashrightarrow S'$. 
\[\xymatrix{ & Z = Z_1 \ar[dl] \ar[dr] \\ S = S_0 \ar@/{}_{1pc}/@{-->}[rrrr]_{f} \ar@{-->}[rr] & &  S_1 \ar@{-->}[rr]^{f_1} & & S'}\]

Furthermore, the minimal resolution $X_1$ of the induced
birational map $\tilde{f}_1:\tilde{S}_{1}\dashrightarrow\tilde{S}'$
between the minimal desingularisations of $\tilde{S}_{1}$ and $\tilde{S}'$
induced by $f_{1}$ is dominated by $X$. More precisely, since $E_0$ is the only possible $(-1)$-curve on $X$
which is exceptional for both induced morphism $X\rightarrow \tilde {S_1}$ and $X \to \tilde{S'}$, $X_1$ is either equal to $X$ if $E_0^2\leq -2$
or is obtained from $X$ by first contracting $E_0$ and then all successive $(-1)$-curves occurring in the minimal resolution of a singular point of $S$ supported on $B_S$ and distinct from the proper base point of $f$ (see Figure \ref{fig:proof} above). 
It follows that the chain associated to $f_{1}\colon S_{1}\dashrightarrow S'$ as defined at the beginning of the proof consists of the curves $E_{1},\ldots,E_{n}={E_0}'$ hence has length $n$. 
We conclude by induction that we can factorize $f$ into exactly $n$ links. \end{proof}

\begin{rem} A by-product of the proof above is the following property
of the intermediate surfaces $Z_{i}$ with two boundary components
that appear in the Theorem:  each one of the boundary component supports
at most one singularity. Note also that neither Lemma \ref{lem:KBneg}
nor the above proof tell something about the possible $K_{Z_{i}}+B_{Z_{i}}$
extremal curves on these intermediate surfaces that do not belong
to the boundary: in the example given in \S \ref{par:inadequacy}
above, we have four $K+B$ extremal rays but only two of them were
supported on the boundary.
\end{rem} 

We introduce a concept that will prove useful in the next section.

\begin{definition}
If $f\colon (S,B_S)\dashrightarrow (S'',B_{S''})$ and $g\colon(S'',B_{S''})\dashrightarrow (S',B_{S'})$ are strictly birational maps of dlt completions, we will say that $f$ and $g$ are in \emph{special position} if $\bs(f^{-1})=\bs(g)$ and in \emph{general position} otherwise.
\end{definition}

It follows in particular from the construction of the factorization $f=f_n\dots f_1 \colon(S,B_S)\dashrightarrow (S',B_{S'})$ given in the proof above that for every $i=1,\dots,n-1$, $f_i$ and $f_{i+1}$ are in general position. 
In general, see Remark \ref{rem:expleref} below, the factorization into elementary links of a composition of two strictly birational maps of dlt completions with irreducible boundaries does not coincide with the concatenation of the factorizations of these maps. 
The following corollary provides however a sufficient condition for this property to hold.
In particular the condition is satisfied when all the surfaces into play are smooth.

\begin{cor} \label{cor:concat}
Let $f:(S,B_{S})\dashrightarrow(S'',B_{S''})$ and $g:(S'',B_{S''})\dashrightarrow(S',B_{S'})$
be birational maps of dlt completions with irreducible boundaries.
If $f$ and $g$ are in general position and at least one of the two
points $\bs(g)$ or $\bs(f^{-1})$ is a smooth point
of $S''$ then the factorization of $g\circ f$ into elementary links given
by Theorem \ref{th:main} is equal to the concatenation of the factorizations of
$f$ and $g$. Furthermore, one has then $\bs(g\circ f)=\bs(f)$ and $\bs((g\circ f)^{-1})=\bs(g^{-1})$.
\end{cor}

\begin{proof}
Up to replacing $f$ and $g$ by their inverses, we may assume that
$\mathcal{B}(g)$ is a smooth point of $S''$.
As before we denote by $\tilde S$ the minimal desingularisation of $S$ (same with $S'$, $S''$) and by $\tilde f$, $\tilde g$ the induced birational maps.
The hypothesis implies that all the base points of $\tilde f^{-1}$ and $ \tilde g$ including infinitely near ones are distinct so that a resolution $\tilde S\stackrel{\sigma}{\leftarrow}X\stackrel{\sigma'}{\rightarrow} \tilde S'$ of the birational map $\tilde S \dashrightarrow \tilde S'$ induced by $g\circ f$ is obtained from $\tilde S''$ by simultaneously resolving the base points of $\tilde f^{-1}$ and $\tilde g$ : 
$$\mygraph{!{<0cm,0cm>;<1cm,0cm>:<0cm,1cm>::}
!{(2,3.3)}*+{X}="X" !{(2,1.5)}*+{\tilde{S}''}="S''tilda"
!{(2,0)}*+{S''}="S''" !{(0,0)}*+{S}="S"!{(4,0)}*+{S'}="S'"
!{(0,1.5)}*+{\tilde{S}}="Stilda" !{(4,1.5)}*+{\tilde{S}'}="S'tilda""X"-@{->}_{\sigma}"Stilda"
"X"-@{->}^{\sigma'}"S'tilda" "X"-@{->}"S''tilda" "S''tilda"-@{->}"S''" "Stilda"-@{->}_{\pi}"S" "S'tilda"-@{->}^{\pi'}"S'" "S"-@{-->}_{f}"S''" "S''"-@{-->}_{g}"S'" "Stilda"-@{-->}_{\tilde f}"S''tilda" "S''tilda"-@{-->}_{\tilde g}"S'tilda"
}$$
The surface $X$ dominates the minimal resolution $X_{f}$ of $\tilde f$ and $X_{g}$ of $\tilde g$.
We denote by $E_0$, $E'_0$, $E''_0$ the strict transforms of $B_S$, $B'_S$ and $B''_S$ in $X$ (or in $X_f$, $X_g$). 
By construction the chain joining $E_{0}$ to $E_{0}'$ in $X$ is the union of the strict transform of the chain joining $E_{0}$ to $E_{0}''$ in $X_{f}$ with the strict transform of the chain joining $E_{0}''$ to $E_{0}'$ in $X_{g}$. 
Since $B_{S''}$ is contracted by $f^{-1}$, its strict transform $E_{0}''$ in $X_{f}$ has negative self-intersection. Furthermore since $\mathcal{B}(g)$ is a smooth point of $S''$, the lift of $g$ to $X_{f}$ has a proper base point on $E_{0}''$ and so the strict transform of $E_{0}''$ in $X$ has self-intersection $\leq-2$. 
Since $E_0''$ is the only curve that could have been a $\left(-1\right)$-curve simultaneously exceptional for $\sigma$ and $\sigma'$, we conclude that $X$ is a minimal resolution of $g\circ f$. 

Now the first part of the assertion follows directly from the construction of the factorization.
The second part follows from the fact that since the image $\bs(g)\in B_{S''}$ of $B_{S'}$ by $g^{-1}$ is distinct from $\bs(f^{-1})$, the image $\bs(g\circ f)$ of $B_{S'}$ by $(g\circ f)^{-1}$ coincides with the image $\bs(f)\in B_S$ of $B_{S''}$ 
by $f^{-1}$. For the same reason, $\bs((g \circ f)^{-1})=\bs(g^{-1})$.
\end{proof}

\begin{rem} \label{rem:expleref}
The assumption that $\mathcal{B}(f^{-1})$ or $\mathcal{B}(g)$ is
a smooth point of $S''$ implies in particular that $B_{S''}$ supports
at most a singular point of $S''$ (Proposition \ref{prop:admissible}, assertion (3)).
So the only situation in which the conclusion of the Corollary above
could fail is when $B_{S''}$ supports exactly two singular points
which are the proper base points of $\mathcal{B}(f^{-1})$ and $\mathcal{B}(g)$ respectively.
The following example, which was pointed out to us by the referee, shows that this phenomenon can indeed occur.

Consider $S = \cpd$, with boundary $B_S$ equal to a line.
We construct a surface $X$ by blowing-up three points: first a point on $S$ producing an exceptional divisor $E$; then the intersection point $E \cap E_0$ (where $E_0$ is the strict transform of $B_S$) producing the exceptional divisor $E'_0$; and finally blowing-up $E_0 \cap E'_0$ producing $E_0''$.

We construct a surface $S''$ from $X$ by contracting the curves $E, E_0$ and $E'_0$; similarly we construct $S'$ by contracting $E, E_0$ and $E_0''$.
These surfaces are singular, we have $\tilde S'' = X$, and $\tilde S'$ is the surface obtained from $X$ by contracting $E_0$ and $E_0''$.
Denote by $f, g$ the birational maps $S \dashrightarrow S''$ and $S'' \dashrightarrow S'$ (see Figure \ref{fig:expleref}). Then the factorization of $g \circ f$ is not the concatenation of the factorizations of $f$ and $g$. 
What's going wrong here is that $X$ is not a minimal resolution of $g\circ f$, indeed $E_0''$ is a $(-1)$-curve on $X$  which is exceptional for both $\sigma$ and $\sigma'$. 
\end{rem}

\begin{figure}[ht]
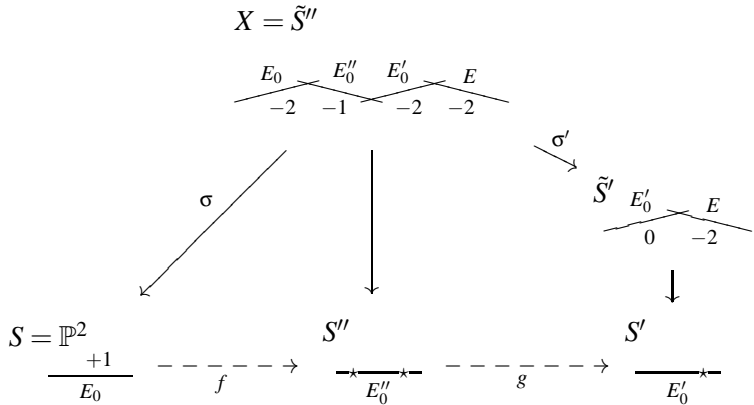

$$\mygraph{
!{<0cm,0cm>;<1cm,0cm>:<0cm,1cm>::}
!{(0,0)  }*+++{\ExpleRefS}="S"
!{(8,0)  }*++{\ExpleRefSprime}="S'" 
!{(4,0)  }*+++{\ExpleRefSdouble}="S''"
!{(4,4)  }*+++{\ExpleRefX}="X"
!{(8,2)  }*++{\ExpleRefSprimetilde}="S'tilde"
"S"-@{<-}^\sigma"X" "S''"-@{<-}"X" "S'tilde"-@{<-}_(0.4){\sigma'}"X" "S'"-@{<-}"S'tilde" 
"S"-@{-->}_f"S''" "S''"-@{-->}_g"S'"
}$$
\caption{The counter-example in Remark \ref{rem:expleref} ($\star$ denotes a singularity, numbers are self-intersections)} \label{fig:expleref}
\end{figure}

\subsection{Additional properties of the factorization}
Noting that the definition of the maximal multiplicity $\lambda$ (see \S \ref{subsec:11}) 
makes sense even when $S$ is not a Mori fiber space, we observe that our algorithm retains one aspect of the log Sarkisov
program of Bruno and Matsuki \cite{BM}, namely the fact that the first divisorial contraction
involved in each link is a maximal extraction:
\begin{prop} \label{prop:max} The birational morphism $Z\to S$
with exceptional divisor $E_{1}$ constructed in the proof of the
theorem is a maximal extraction. 
\end{prop}
\begin{proof} A maximal extraction (see \cite[prop. 13-1-8]{Mat}
and \cite[p. 485]{BM} for the logarithmic case) is obtained from
a smooth surface which dominates $S$ and $S'$ by a process of the
$K+B$-MMP. So we may use the surface $X$ from the proof of the theorem.
The precise procedure consists in two steps (we use the notations $\lambda$
and $H$ that have been defined in \S \ref{subsec:11}): 
Running first a $K+B+\frac{1}{\lambda}H$-MMP over $S$ until we reach a log minimal model, then running a $K+B$-MMP over $S$; the last contraction gives a maximal extraction. 
The crucial observation is that each extremal divisorial contraction
of the log MMP in the first step is also a contraction for the genuine
$K+B$-MMP. The fact that we are running a log MMP over $S$ guarantees
that the only curves affected by the procedure are contained in the
boundary. By Lemma \ref{lem:KBneg}, as long as $E_{1}$
admits two neighboring components ($E_{0}$ and another one), it cannot
correspond to a $K+B$ negative extremal ray. Remark also that if
$B_{S}$ supports a singularity $q$ which is not a proper base point
for $f$, then all exceptional divisors of the resolution of $q$
have multiplicities $\lambda_{i}=0$ and thus are contracted in the
first step. It follows that the maximal extraction we constructed, which is the last
divisorial contraction $Z\rightarrow S$, must have $E_{1}$ as exceptional
divisor. \end{proof}

\begin{rem} In contrast with the log Sarkisov algorithm of Bruno and Matsuki, we did not assume from the beginning
that the pairs $(S,B_{S})$ and $(S',B_{S'})$ were log-MMP related. 
In our situation, this property is automatic: this is probably a well-known fact, but we can also obtain it as a by-product of the proof of Theorem \ref{th:main}. 
Indeed, letting again $E_{0},\dots,E_{n}$
be the subchain of rational curves in the boundary
$B_{X}$ of $X$ defined in the proof, Lemma
\ref{lem:KBneg} guarantees that all the irreducible components of
$B_{X}$ except the ones contained in that chain can be successively
contracted by a process of the $K+B$ MMP. The surface $W$ obtained
by this procedure has boundary $B_{W}=\sum_{i=0}^{n}E_{i}$ and dominates
both $S$ and $S'$ by a sequence of $K+B$ divisorial contractions.

On the other hand, the elementary example of the identity map of $\mathbb{A}^{2}$
viewed as a rational map from $\mathbb{P}^{1}\times\mathbb{P}^{1}\dashrightarrow\mathbb{P}^{2}$
with a unique proper base point $p$ located at the intersection of
the two rules at infinity and for which the blow-up of $p$ is not a
$K+B$ extremal contraction shows that arbitrary dlt completions 
of a given quasi-projective surface need not be log-MMP related in general. 
So if one wants to extend our factorization result to pairs with reducible boundaries, 
it becomes necessary to at least require from the very beginning that the pairs under 
consideration are log-MMP related. 
\end{rem}

\section{Quasi-projective surfaces with smooth completions} \label{sec:3}\label{sec:smooth}

In this section  we derive from our factorization theorem a general description of the automorphism group of $V$ when $V$ admits a smooth completion $(S,B_S)$ with irreducible boundary $B_S\simeq \mathbb{P}^1$. 
In what follows, such pairs $(S,B_S)$ are simply referred to as \emph{smooth completions}, the inner smooth quasi-projective surface $V=S\smallsetminus B_S$ being implicit. 

Smooth completions $(S,B_S)$ for which $B_S^2<0$ can be quite arbitrary since for instance any blow-up $\sigma : 
S\rightarrow S'$ of a point on a smooth projective surface $S'$ with exceptional divisor $B_S$ gives rise to such a pair $(S,B_S)$. 
In contrast, the possible structures of pairs $(S,B_S)$ with $B_S^2\geq 0$ are much more constrained, as summarized by the following Proposition:

\begin{prop} \label{prop:models} If $(S,B_S)$ is a smooth completion with $B_S^2\geq 0$, then after the contraction of finitely many $(-1)$-curves contained in $V$, we reach a pair of the following type:
\begin{enumerate}
\item $(\mathbb{P}^2,B)$ where $B$ is either a line or a smooth conic,
\item $(\mathbb{F},B)$ where $p\colon \mathbb{F}\rightarrow D$ is a ruled surface over a smooth projective curve $D$ and where $B$ is either a fiber or a section. 
Furthermore, if $B^2\neq 0$ then $\mathbb{F}$ is a Hirzebruch surface $\mathbb{F}_n$, for some $n\geq 0$, and $B$ is a section.
\end{enumerate}
\end{prop}

\begin{proof} Up to replacing $(S,B_S)$ by a pair obtained by successively contracting all possible $(-1)$-curves in $S\smallsetminus B_S$ and having the strict transform of $B_S$ for its boundary, we may assume from the very beginning that $S\smallsetminus B_S$ does not contain a $(-1)$-curve. 
Since $(K_S+B_S)\cdot B_S=-2$ by adjunction formula, it follows that $K_S+B_S$ is not nef and so there exists a $K_S+B_S$-negative extremal rational curve $C$ on $S$. Since $B_S^2\geq 0$, the conditions $(K_S+B_S)\cdot C=(K_S+C)\cdot C+B_S\cdot C-C^2=-2+B_S\cdot C-C^2<0$ and $C^2<0$ would imply that $C$ is a $(-1)$-curve disjoint from $B_S$, which is impossible.
Thus $(S,B_S)$ is a log Mori fiber space $g:S\rightarrow D$. 
If $D$ is a point then $S$ is smooth log del Pezzo surface of rank $1$, whence is isomorphic to $\mathbb{P}^2$ and $B_S \simeq \mathbb{P}^1$ is either a line or a smooth conic. Otherwise, since $S$ is smooth $g:S\rightarrow D$ is a $\mathbb{P}^1$-bundle and the condition $(K_S+B_S)\cdot F=(K_S+F)\cdot F+B_S\cdot F=-2+B_S\cdot F<0$ for a fiber $F\simeq \mathbb{P}^1$ of $g$ implies that $B_S\cdot F=0 \textrm{ or } 1$. 
Thus $B_S$ is either  a fiber if $B_S\cdot F=0$ or a section otherwise. 
This immediately implies the remaining assertions.
\end{proof}

\subsection{Triangular birational maps between smooth completions}

Let us first observe that if $(S,B_S)$ is a smooth completion with $B_S^2<0$ then every birational map of smooth completions $f:(S,B_S)\dashrightarrow (S',B_{S'})$ is in fact an isomorphism. 
Indeed, otherwise it would have a proper base point on $B_S$, and since $B_S^2<0$ it would follow that the total transform of $B_S$ in the minimal resolution of $f$ contains no $(-1)$-curve except the strict transform of $B_{S'}$, in contradiction with the fact that $S'$ is smooth. 
It follows in particular that if a smooth quasi-projective surface $V$ admits a smooth completion $(S,B_S)$ with $B_S^2<0$ then the automorphism group of $V$ coincides with the subgroup ${\rm Aut}(S,B_S)$ of ${\rm Aut}(S)$ consisting of automorphisms preserving the boundary $B_S$. 
In contrast, if $(S,B_S)$ and $(S',B_{S'})$ are smooth completions with $B_S^2\geq 0$ or $B_{S'}^2\geq 0$, then strictly birational maps of smooth completions $(S,B_S)\dashrightarrow (S',B_{S'})$ may exist in general.

\subsubsection{Structure of intermediate pairs} Given such a strictly birational map, we prove in the next lemma that the dlt pairs $(S_i,B_{S_i})$ which appear in the factorization of $f$ as in Theorem \ref{th:main} have at most one singularity. 
So the following definition makes sense: If $S_i$ is singular, then we say that it has \emph{index} $k$ if in the minimal resolution of its singularities the exceptional curve which intersects the strict transform of $B_{S_i}$ has self-intersection $-k$. Otherwise, if $S_i$ is smooth then we say that $S_i$ has index 1. 
We note $\ind(S_i)$ the index of $S_i$.

\begin{lem} \label{lem:triangle}
Let $f:(S,B_S)\dashrightarrow (S',B_{S'})$ be a strictly birational map of smooth completions and let $S = S_0\dashrightarrow S_1 \dashrightarrow \dots \dashrightarrow S_n = S'$ be its  factorization into elementary links given by Theorem \ref{th:main}.  Then the following holds :

1) If $B_S^2=0$ then each $S_i$ is smooth with $B_{S_i}^2=B_S^2=0$,

2) If $B_S^2>0$ then each $S_i$ has at most one singularity. Furthermore:

$\quad$ a) If $S_i$ is smooth then $B_{S_i}^2=B_S^2$ whereas if $S_i$ is singular, the boundary of a minimal resolution of $S_i$ is a chain of $B_S^2 + 1$ rational curves with self-intersections $(0,-k_i,-2, \dots,-2)$ where $k_i = \ind(S_i)$;

$\quad$ b)  For all $i = 0, \dots, n-1$ the indexes of $S_i$ and $S_{i+1}$ differ exactly by 1 and if $\ind(S_i) \ge 2$ and $\ind(S_i)  = \ind(S_{i-1}) - 1$ then   $\ind(S_{i+1}) = \ind(S_i) - 1$.
\end{lem}

\begin{proof} Let $(S_j,B_j)$ be one of the intermediate dlt completions, and let $f_j:(S_j,B_j)\dashrightarrow (S',B_{S'})$ be the induced birational map. Suppose $S_j$ is smooth, with $B_{S_j}^2 =B_S^2=d \geq 0$ and consider as in the proof of Theorem \ref{th:main} the surface $Y$ containing the strict transforms $E_j$ and $E_{j+1}$ of the boundaries of $S_j$ and $S_{j+1}$. 
Since $S'$ is smooth, the strict transform of $B_j$ in the minimal resolution $X_j$ of $f_j$ is a $(-1)$-curve. It follows that the boundary of $Y$ is equal to a chain of $d + 2$ curves with self-intersections $(-1,-1,-2,\dots,-2)$. If $d=0$ then $S_{j+1}$ is again smooth with $B_{S_{j+1}}^2=0$ and so, 1) follows by induction. 
Otherwise, if $d>0$ then $S_{j+1}$ has a unique singularity and the boundary of the minimal resolution of $S_{j+1}$ is a chain of $d + 1$ curves with self-intersections $(0,-2,\dots,-2)$. 
In particular, $S_{j+1}$ has index 2 (see Figure \ref{fig:triang}, (a), with $k = 2$).
Now we proceed by induction, assuming that $S_i$ has exactly one singularity, and that the boundary of the minimal resolution $\tilde{S}_i$ of $S_i$ is a chain of $d + 1$ rational curves with self-intersections $(0,-(k-1),-2, \dots,-2)$, where $k-1 = \ind(S_i) \ge 2$. 
We denote by $C$ the second irreducible component of this chain which has thus self-intersection $-(k-1)$. 
Let $\tilde{S}_i\leftarrow X_i \rightarrow S'$ be the minimal resolution of the induced birational map $\tilde{S}_i \dashrightarrow S'$. 
Since the strict transform $E_i$ of $B_{S_i}$ is a $0$-curve on $\tilde{S}_i$ and a $(-1)$-curve on $X_i$ as $S'$ is smooth, we see that there is exactly one blow-up on $E_i$, which by definition produces the divisor $E_{i+1}$. Then there are two cases :

a) If the proper base point on $E_i$ coincides with the intersection point of $E_i$ and $C$, then the  boundary of $Y$ is a chain of curves with self-intersections $(-1, -1, -k, -2, \dots, -2)$, where the first three are $E_i$ and $E_{i+1}$ and $C$. 
Thus in this case $S_{i+1}$ has again exactly one singularity and has index $k$ (the picture is again Figure \ref{fig:triang}, (a)). 

b) Otherwise, if the proper base point on $E_i$ is any other point, then the  boundary of $Y$ is a chain of curves with self-intersections $(-2, \dots, -2, -(k-1), -1, -1)$, where the last three are $C$,  $E_i$ and $E_{i+1}$. 
In this case $S_{i+1}$ has again at most one singularity and has index $k-2$ (see Figure \ref{fig:triang}, (b)). 
It is smooth if and only if $k-1=2$ and in this case its boundary $B_{S_{i+1}}$, which is the strict transform of $E_{i+1}$ has again self-intersection $B_{S_{i+1}}^2=d$.

\begin{figure}[ht]
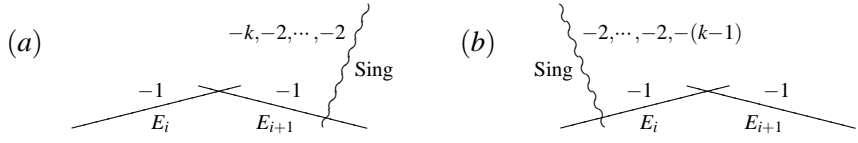
 
$$ (a) \quad \dessinTa \qquad (b) \quad \dessinTb $$
\caption{Boundary of $Y$ in the proof of Lemma \ref{lem:triangle}.} 
\label{fig:triang}
\end{figure}

The last assertion follows from the fact that by construction the center of the blow-up on $E_{i+1}$ producing the next divisor $E_{i+2}$ does not coincide with the intersection point of $E_i$ and $E_{i+1}$.
\end{proof}

\subsubsection{Triangular birational maps}\label{Tr-disc}
\begin{definition} \label{def:triang} A strictly birational map of smooth completions $\phi:(S,B_S)\dashrightarrow (S',B_{S'})$  is called \emph{triangular} if all the intermediate surfaces $S_i$ that appear in the factorization produced by Theorem \ref{th:main} are singular. 
 \end{definition}

Given a smooth pair $(S,B_S)$ with $B_S^2\geq 0$, it follows from Proposition \ref{prop:models} that $S$ dominates birationally a surface $\F$ which is either $\mathbb{P}^2$ or a ruled surface. 

First we discuss the case where $B_S^2 = 0$.
Then the strict transform of $B_S$ in $\F$ still have self-intersection $0$, so $\F$ is a ruled surface $p\colon\mathbb{F}\rightarrow D$ and the strict transform of $B_S$ is either a fiber $F$ or a section of $p$. 
Note that in the second case $\mathbb{F}$ is isomorphic to $\mathbb{P}^1\times \mathbb{P}^1$ in such a way that $p$ coincides with the first projection while the strict transform of $B_S$ is a fiber $F$ of the second projection: up to changing the projection  we can assume that $B_S$ is a fiber, as in the first case. 
Then, it follows from Lemma \ref{lem:triangle} that the notion of a triangular map coincides with that of a link and that every such link consists of the blow-up of a point on $F$ followed by the contraction of its strict transform. 
Assume now that $f\colon(S,B_S)\dashrightarrow (S',B_{S'})$  is a strictly birational map of smooth completions, where $(S,B_S)$ and $(S',B_{S'})$ dominate some ruled surfaces $p\colon\mathbb{F}\rightarrow D$ and $p':\mathbb{F}'\rightarrow D'$ respectively.
Then $B_{S'}^2 = 0$, we can assume that the strict transforms of $B_S$ and $B_{S'}$ are fibers of $p$ and $p'$ respectively, and  the birational transformation $\mathbb{F}\dashrightarrow \mathbb{F}'$ induced by $f$ consists of elementary transformations between ruled surfaces. 
It follows that $f$ preserves the $\mathbb{P}^1$-fibrations $\rho:S\rightarrow D$ and $\rho':S'\rightarrow D'$ induced by these rulings  hence induces an isomorphism $f:V=S\smallsetminus B_S\rightarrow V'=S' \smallsetminus B_{S'}$ of $\mathbb{P}^1$-fibered quasi-projective surfaces \[\xymatrix{ V=S\smallsetminus B_S \ar[r]^-{\sim}_-{f} \ar[d]_{\rho\mid_V} & V'=S'\smallsetminus B_{S'} \ar[d]^{\rho'\mid_{V'}}\\ D\smallsetminus \rho(B_S)  \ar[r]^{\sim} & D'\smallsetminus \rho'(B_{S'}).}\]

Next we consider the case of a triangular map  $\phi\colon(S,B_S)\dashrightarrow (S',B_{S'})$ between smooth completions with $B_S^2=B_{S'}^2=d>0$. Note that since $B_S$ and $B_{S'}$ are smooth rational curves, it follows form Noether's Lemma that the surfaces $S$ and $S'$ are rational. We deduce from the description given in the proof of Lemma \ref{lem:triangle} that the total transform of $B_S$ in the minimal resolution $X$ of $\phi$ is a tree of rational curves with the dual graph pictured in Figure \ref{fig:Minres}.

\begin{figure}[ht]
\begin{pspicture}(0,-0.1)(8,2.3)
\psline(5,1)(8,1)
\psframe(8,0.75)(9.5,1.25)\rput(8.75,1){{\scriptsize $d -1$}}
\rput(7,1){\textbullet}\rput(7,1.5){{\small $C$}}\rput(7.05,0.7){{\scriptsize $-k$}}
\rput(5,1){\textbullet}\rput(5.05,0.7){{\scriptsize $-2$}}
\psline(4.5,0.5)(5,1)\psline(3.5,0.5)(4.5,0.5)
\psframe(2,0.75)(3.5,0.25)\rput(2.75,0.5){{\scriptsize $k -2$}}
\psline(1,0.5)(2,0.5)
\rput(1,0.5){\textbullet}\rput(1,0.8){{\small $E_0$}}\rput(1.05,0.2){{\scriptsize $-1$}}
\psline(4.5,1.5)(5,1)\psline(3.5,1.5)(4.5,1.5)
\psframe(2,1.25)(3.5,1.75)\rput(2.75,1.5){{\scriptsize$k -2$}}
\psline(1,1.5)(2,1.5)
\rput(1,1.5){\textbullet}\rput(1,1.8){{\small ${E_0}'$}}\rput(1.05,1.2){{\scriptsize $-1$}}

\pnode(5.5,0.3){A}\pnode(1.4,0.3){B}\ncbar[angle=-90]{A}{B}
\ncput*{$H$}
\pnode(5.5,1.7){C}\pnode(1.4,1.7){D}\ncbar[angle=90]{C}{D}
\ncput*{$H'$}
\pnode(7.9,1.4){E}\pnode(9.6,1.4){F}\ncbar[angle=90]{E}{F}
\ncput*{$D$}
\end{pspicture}
\caption{Minimal resolution of a triangular map.}
\label{fig:Minres}
\end{figure}
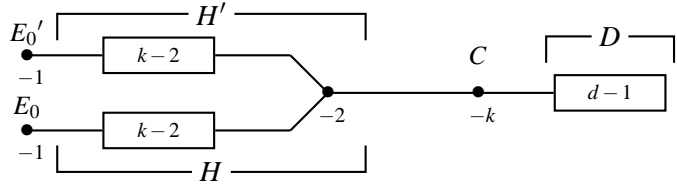

Here $E_0$ and ${E_0}'=E_n$ denote the strict transforms of $B_S$ and $B_{S'}$ respectively, the two boxes on the left represent chains of $k-2$ rational curves with self-intersection $\left(-2\right)$, and the one on the right a chain $D$ of $d-1$ such curves.  
Note also that the proper base point of $\phi$ coincides with the proper base point of the first elementary link $S=S_0\dashrightarrow S_1$ while the one of its inverse coincides with the proper base point of the inverse of the last one $S_{n-1}\dashrightarrow S_n=S'$ (see Corollary \ref{cor:concat}).  

Let $\delta:X\rightarrow \hat{S}$ and $\delta':X\rightarrow \hat{S}'$ be the morphisms given by the smooth contractions of the sub-trees $H\cup H'\cup {E_0}'$ and $H\cup H'\cup E_0$ onto $q=E_0\cap C$ and $q'={E_0}'\cap C$. 
Since $S$ and $S'$ are rational and $E_0^2=0$ and ${E_0'}^2=0$ on $\hat{S}$ and $\hat{S'}$, it follows from Riemann-Roch Theorem that the complete linear systems $|E_0|$ and $|{E_0}'|$ are base point free and define $\mathbb{P}^1$-fibrations $\hat{\rho}\colon\hat{S}\rightarrow \mathbb{P}^1$ and $\hat{\rho}'\colon\hat{S}'\rightarrow \mathbb{P}^1$ both having the image of $C$ as a section. 
Note further that the image of $D$ in $\hat{S}$ and $\hat{S}'$ is a proper subset of a fiber of $\hat{\rho}$ and $\hat{\rho}'$ respectively: indeed, if not empty, the image of $D$ has negative definite self-intersection matrix and hence cannot be equal to a full fiber of a $\mathbb{P}^1$-fibration. By contracting the remaining exceptional divisors, we see that $|E_0|$ (resp. $|{E_0}'|$) coincides with the strict transform on $\hat{S}$ (resp. $\hat{S'}$) of the rational subpencil $\mathcal{P}_{\bs(\phi)}\subset |B_S|$ (resp. $\mathcal{P'}_{\bs(\phi^{-1})}\subset |B_{S'}|$) consisting of curves having local intersection number with $B_S$ (resp. $B_{S'}$) at $\bs(\phi)$ (resp. $\bs(\phi^{-1})$) equal to $d$. 
Equivalently, the fibrations $\hat{\rho}$ and $\hat{\rho}'$ coincide respectively with the minimal resolution of the rational maps $\rho\colon S\dashrightarrow \mathbb{P}^1$ and $\rho'\colon S'\dashrightarrow \mathbb{P}^1$ defined by $\mathcal{P}_{\bs(\phi)}$ and $\mathcal{P'}_{\bs(\phi^{-1})}$. 
These two maps restrict on $V=S\smallsetminus B_S$ and $V'=S'\smallsetminus B_{S'}$ to quasi-projective $\mathbb{A}^1$-fibrations $\rho\mid_V\colon V\rightarrow \mathbb{A}^1=\mathbb{P}^1\smallsetminus \rho(B_S)$ and $\rho'\mid_{V'}\colon V'\rightarrow \mathbb{A}^1=\mathbb{P}^1\smallsetminus \rho'(B_{S'})$, i.e., surjective morphisms with general fiber isomorphic to the affine line $\mathbb{A}^1$.

The birational map $\phi\colon S\dashrightarrow S'$ lifts to $\hat{\phi}\colon\hat{S}\dashrightarrow \hat{S}'$ mapping $\hat{S}\smallsetminus E_0$ isomorphically onto $\hat{S}'\smallsetminus {E_0}'$, having $q$ as unique proper base point while its inverse has $q'$ as a unique proper base point. 
Since the total transforms of $E_0$ and $E_0'$ in $X$ coincide, the lifted $\mathbb{P}^1$-fibrations $\hat{\rho}\circ \delta$ and $\hat{\rho}'\circ \delta'$ on $X$ coincide. 
This implies that $\hat{\phi}$ restricts to an isomorphism of $\mathbb{A}^1$-fibered quasi-projective surfaces 
\[\xymatrix{ V=S\smallsetminus B_S=\hat{S}\smallsetminus \tau^{-1}(B_S) \ar[r]^-{\sim}_-{\hat{\phi}} \ar[d]_{\hat{\rho}\mid_V} & V'=S'\smallsetminus B_{S'}=\hat{S'}\smallsetminus {\tau'}^{-1}(B_{S'}) \ar[d]^{\hat{\rho}'\mid_{V'}}\\ \mathbb{A}^1 \ar[r]^{\sim} & \mathbb{A}^1}\] 
where $\tau:\hat{S}\rightarrow S$ and $\tau':\hat{S}'\rightarrow S'$ denote the contraction of $C$ and the right chain $D$ of $d-1$ curves with self-intersection $-2$ pictured in Figure \ref{fig:Minres} above.

A birational map $\hat{\phi}\colon\hat{S}\dashrightarrow \hat{S}'$ restricting to an isomorphism of $\mathbb{A}^1$-fibered surfaces as above is called a \emph{fibered modification} (see also \cite[2.2.1]{BD}).

In general, if $(S, B_S)$ is a smooth completion with $B_S^2=d>0$ and $p$ is a point of $B_S$ then the base locus of the linear subsystem $\mathcal{P}_p\subset |B_S|$ consisting of curves having a local intersection number with $B_S$ at $p$ equal to $d$ is solved as follows.
We perform $d$ successive blow-ups with centers on the successive strict transforms of $B_S$, until we reach a surface $\hat{S}$ on which the strict transform of $\mathcal{P}_p$ is equal to the complete linear system $|E_0|$ generated by the strict transform $E_0$ of $B_S$. 
Since $E_0$ is a smooth rational curve with $E_0^2 = 0$, $\mathcal{P}_p$ defines a rational pencil $\rho_p\colon S\dashrightarrow \mathbb{P}^1$ which restricts on $V=S\smallsetminus B_S$ to a quasi-projective $\mathbb{A}^1$-fibration $\rho_p\mid_V\colon V\rightarrow \mathbb{A}^1$. 
This leads to the following alternative characterization of triangular maps:

\begin{lem}\label{lem:TrCarac} For a strictly birational map of smooth completions  $\phi:(S,B_S)\dashrightarrow (S',B_{S'})$ with $B_S^2=B_{S'}^2>0$, the following are equivalent :
\begin{enumerate}[a)]
\item $\phi$ is a triangular map;
\item There exist points $p\in B_S$ and $p'\in B_{S'}$ such that $\phi$ maps the pencil $\mathcal{P}_p$ onto the pencil $\mathcal{P}'_{p'}$; If so, the points $p$ and $p'$ are equal to $\bs(\phi)$ and $\bs(\phi^{-1})$ respectively.
\item $\phi$ maps the pencil $\mathcal{P}_{\bs(\phi)}$ onto the pencil $\mathcal{P}'_{\bs(\phi^{-1})}$;
\item $\phi$ induces an isomorphism of $\mathbb{A}^1$-fibered quasi-projective surfaces $(S\smallsetminus B_S,\rho_{\bs(\phi)})\stackrel{\sim}{\rightarrow}(S'\smallsetminus B_{S'},\rho'_{\bs(\phi^{-1})})$.
\end{enumerate}
\end{lem}

\begin{proof} Properties c) and d) are clearly equivalent. If b) holds and the proper base point of $\phi$ is distinct from $p$ then all infinitely near base points of $\phi$ are also distinct from $p$. 
Since $\phi$ contracts $B_S$, the strict transform in $S'$ of a general member of $\mathcal{P}_p$ has self-intersection strictly bigger than $B_S^2=B_{S'}^2$ hence cannot be a general member of a pencil of the form ${\mathcal{P}'}_{p'}$. 
So $\bs(\phi)=p$ and for the same reason $\bs(\phi^{-1})=p'$ which proves the equivalence of b) and c). 
The fact that a triangular map  $\phi:(S,B_S)\dashrightarrow (S',B_{S'})$ maps $\mathcal{P}_{\bs(\phi)}$ onto ${\mathcal{P}'}_{\bs(\phi^{-1})}$ follows from the above discussion. 

It remains to prove that c) implies a). 
If c) holds then since $S$ and $S'$ are both smooth, the strict transforms of $B_S$ and $B_{S'}$ in the minimal resolution $S\stackrel{\sigma}{\leftarrow} X \stackrel{\sigma'}{\rightarrow} S'$ of $\phi$ are both $(-1)$-curves. 
So $\phi$ and $\phi^{-1}$ both have at least $d+1$ base points including infinitely near ones and their first $d+1$ base points are supported on $B_S$ and $B_{S'}$ respectively. 
This implies in turn that $\sigma$ and $\sigma'$ factor respectively through the minimal resolutions $\pi\colon \hat{S}\rightarrow S$ and $\pi'\colon\hat{S}'\rightarrow S'$ of the base points of $\mathcal{P}_{\bs(\phi)}$ and $\mathcal{P'}_{\bs(\phi^{-1})}$ and that the induced birational map $\hat{\phi}:\hat{S}\dashrightarrow \hat{S}'$ is a fibered modification. 
By virtue of \cite[2.2.4]{BD}, the dual graph of the total transform of $B_S$ in $X$ looks like the one pictured in Figure \ref{fig:Minres} for which it is straightforward to check that all intermediate surfaces occurring in the decomposition are singular. Thus $\phi$ is a triangular map. 
\end{proof}

The following Corollary, which is an immediate consequence of the previous Lemma, will be frequently used in the sequel:

\begin{cor}\label{cor:comp} 
If $\phi\colon (S,B_S)\dashrightarrow (S'',B_{S''})$ and $\phi''\colon (S'',B_{S''})\dashrightarrow (S',B_{S'})$ are triangular maps of smooth completions with $B_S^2>0$ and $\phi, \phi''$ in special position, then the composition $\phi'=\phi'' \circ \phi\colon (S,B_S)\dashrightarrow (S',B_{S'})$ is either an isomorphism of pairs mapping $\bs(\phi)$ on $\bs({\phi''}^{-1})$ or a triangular map with $\bs(\phi')=\bs(\phi)$ and $\bs({\phi'}^{-1})=\bs({\phi''}^{-1})$.
\end{cor}

\subsection{Automorphisms of quasi-projective surfaces with smooth completions} 

\subsubsection{Decomposition into triangular maps and normal forms}
 
Given a strictly birational map $f:(S,B_{S})\dashrightarrow(S',B_{S'})$
of smooth completions with $B_{S}^{2}=B_{S'}^{2}>0$, Lemma \ref{lem:triangle} provides
a decomposition of $f$ into a finite sequence 
\[
f=\phi_{n}\dots\phi_{1}\colon (S,B_{S})=(S_{0},B_{S_{0}})\stackrel{\phi_{1}}{\dashrightarrow}(S_{1},B_{S_{1}})\stackrel{\phi_{2}}{\dashrightarrow}\dots\stackrel{\phi_{n}}{\dashrightarrow}(S_{n},B_{S_{n}})=(S',B_{S'})
\]
of triangular maps between smooth completions. Such a decomposition
of $f$ is called \emph{minimal} if there does not exist any other decomposition
with strictly less than $n$ terms. The following Proposition provides
a characterization of these minimal decompositions.

\begin{prop} \label{prop:main} A composition 
\[
f=\phi_{n}\dots\phi_{1} \colon (S,B_{S})=(S_{0},B_{S_{0}})\stackrel{\phi_{1}}{\dashrightarrow}(S_{1},B_{S_{1}})\stackrel{\phi_{2}}{\dashrightarrow}\dots\stackrel{\phi_{n}}{\dashrightarrow}(S_{n},B_{S_{n}})=(S',B_{S'})
\]
of triangular maps between smooth completions with $B_{S}^{2}=B_{S'}^{2}>0$
is minimal if and only if for every $i=1,\ldots,n-1$, the maps $\phi_{i}$ and $\phi_{i+1}$
are in general position. 

Furthermore, if these conditions are satisfied, then the following
holds:

a) The map $f$ is strictly birational with $\mathcal{B}(f)=\mathcal{B}(\phi_{1})$ and
$\mathcal{B}(f^{-1})=\mathcal{B}(\phi_{n}^{-1})$,

b) For every other minimal decomposition
\[
f=\phi_{n}'\dots\phi_{1}' \colon (S,B_{S})=(S_{0}',B_{S_{0}'})\stackrel{\phi_{1}'}{\dashrightarrow}(S_{1}',B_{S_{1}'})\stackrel{\phi_{2}'}{\dashrightarrow}\dots\stackrel{\phi_{n}'}{\dashrightarrow}(S_{n}',B_{S_{n}'})=(S',B_{S'})
\]
of $f$ there exists isomorphisms of
pairs $\alpha_0=\mathrm{id}_S$, $\alpha_{i} \colon (S_{i},B_{S_{i}})\stackrel{\sim}{\rightarrow}(S_{i}',B_{S_{i}'})$,
$i=1,\ldots n-1$ and $\alpha_n=\mathrm{id}_{S'}$, such that $\alpha_{i}\phi_{i}=\phi_{i}'\alpha_{i-1}$
for every $i=1,\dots, n$. 
\end{prop}

\begin{proof} 
First note that by virtue of Corollary \ref{cor:comp}, the composition of two triangular maps
in special position is either triangular or an isomorphism of pairs.
Therefore a composition $\phi_{n}\dots\phi_{1}$ in which for some
$i$ the maps $\phi_{i}$ and $\phi_{i+1}$ are in special position
cannot be minimal. 

Next assume that $f=\phi_n\cdots \phi_1$ is a composition for which any two successive triangular maps are in general position. To prove a), up to changing $f$ with its inverse, it is enough to check that $f$ is strictly birational with $\mathcal{B}(f^{-1})=\mathcal{B}(\phi_{n}^{-1})$.
We proceed by induction on $n$, the case $n=1$ being obvious. 
If $n>1$ then by induction hypothesis $f_{n-1}=\phi_{n-1}\dots\phi_{1}$ is a strictly birational map which contracts the curve $B_{S_{0}}$ to the proper base point $p\in B_{S_{n-1}}$ of $\phi_{n-1}^{-1}$.
The curve $B_{S_{n-1}}$ is contracted in turn by $\phi_{n}$ onto
the proper base point of $\phi_{n}^{-1}$. But since $\phi_{n-1}$
and $\phi_{n}$ are in general position, $p$ is not a base point
of $\phi_{n}$ and so, $f=\phi_{n}f_{n-1}$ contracts $B_{S_{0}}$
onto $\phi_{n}(B_{S_{n-1}})=\mathcal{B}(\phi_{n}^{-1})$. This shows
that $f$ is strictly birational and that $\mathcal{B}(f^{-1})=\mathcal{B}(\phi_{n}^{-1})$.

Now let $f=\phi_{m}'\dots\phi_{1}'\colon(S,B_{S})\dashrightarrow(S',B_{S'})$, $m\leq n$,
be a minimal decomposition of $f$ into triangular maps. 
By Corollary \ref{cor:comp} again, any two successive triangular maps must be in general position. 
If $\phi_{1}^{-1}$ and $\phi_{1}'$ were in general position, then by a) $\phi_{m}'\dots \phi_{1}'\phi_{1}^{-1} \dots\phi_{n}^{-1}$ would be a strictly birational map restricting to the identity on $S'\setminus B_{S'}$, whence on  $S'$, which is absurd.  
Therefore, $\phi_{1}^{-1}$ and $\phi_{1}'$ are in special
position and it follows from Corollary \ref{cor:comp} that $\alpha_{1}=\phi_{1}'\phi_{1}^{-1}\colon(S_{1},B_{S_{1}})\dashrightarrow(S_{1}',B_{S_{1}'})$
is either a triangular map or an isomorphism of pairs. 
But if $\alpha_{1}$ is triangular, then, again by Corollary \ref{cor:comp}, we would have $\mathcal{B}(\alpha_{1})=\mathcal{B}(\phi_{1}^{-1})$ and $\mathcal{B}(\alpha_{1}^{-1})=\mathcal{B}((\phi_{1}')^{-1})$.
The pairs of maps $\alpha_{1}, \phi_{2}'$ and $\phi_{2}^{-1}, \alpha_{1}$
would then be both in general position and $\phi_{m}'\dots\phi_{2}'\alpha_{1}\phi_{2}^{-1}\dots\phi_{n}^{-1}$
would again be strictly birational. So $\alpha_{1}$ is an isomorphism
of pairs and writing $\psi_{2}'=\phi_{2}'\alpha_{1}$, which is again
a triangular map, we deduce in a similar way that $\psi_{2}'\phi_{2}^{-1}\colon(S_{2},B_{S_{2}})\dashrightarrow(S_{2}',B_{S_{2}'})$
is an isomorphism, that we denote by $\alpha_{2}$. By induction,
we define $\psi_{r}'=\phi_{r}'\alpha_{r-1}$ and obtain an isomorphism
$\alpha_{r}=\psi_{r}'\phi_{r}^{-1}\colon(S_{r},B_{S_{r}})\dashrightarrow(S_{r}',B_{S_{r}'})$
for $r=2,\ldots,m$. The last relation obtained is $\alpha_{m}\phi_{m+1}^{-1}\dots\phi_{n}^{-1}=\mathrm{id}_{S'}:(S',B_{S'})\dashrightarrow(S',B_{S'})$
from which we deduce that $m=n$, and $\alpha_{n}=\mathrm{id}_{S'}$.
Choosing $\alpha_{0}=\mathrm{id}_{S}$ we find that $\alpha_{i}\phi_{i}=\phi_{i}'\alpha_{i-1}$
for every $i=1,\ldots,n$. 

This proves on the one hand that the decomposition $f=\phi_n\cdots \phi_1$ was minimal and that b) holds for this decomposition. 
\end{proof}

\begin{definition}\label{def:len} The number $\ell(f)$ of triangular maps occurring in a minimal decomposition of a birational map $f:(S,B)\dashrightarrow (S',B_{S'})$ of smooth completions is called the \emph{length} of $f$. 
\end{definition}

\begin{cor} \label{cor:cancel} Let $f=\phi_{n}\dots\phi_{1}\colon (S,B_{S})\dashrightarrow(S',B_{S'})$
be a strictly birational composition of $n \ge 2$ triangular maps. If $\mathcal{B}(f)\neq\mathcal{B}(\phi_{1})$
then there exists an index $i\in\left\{ 2,\ldots,n-1\right\} $ such
that $\phi_{i}\dots\phi_{1}$ is an isomorphism.
\end{cor}
\begin{proof}
We proceed by induction on the number of triangular maps in the composition. If $n=2$
then if $\phi_{2}$ and $\phi_{1}$ are in general position or if
$\phi_{2}$ and $\phi_{1}$ are in special position and $\phi_{2}\phi_{1}$
is triangular then $\mathcal{B}(\phi_{2}\phi_{1})=\mathcal{B}(\phi_{1})$
by Proposition \ref{prop:main} and Corollary \ref{cor:comp}. So $\mathcal{B}(\phi_{2}\phi_{1})=\mathcal{B}(\phi_{1})$
unless $\phi_{2}\phi_{1}$ is an isomorphism. 

Now suppose that $n>2$. 

If $\phi_{2}$ and $\phi_{1}$ are in special position then either
$\phi_{2}\phi_{1}$ is an isomorphism and we are done, or $\phi_{2}'=\phi_{2}\phi_{1}$
is a triangular map with proper base point equal to that of $\phi_{1}$.
Since $f=\phi_{n}\dots\phi_{2}'$ with $\mathcal{B}(f)\neq\mathcal{B}(\phi_{2}')$
the induction hypothesis implies that there exists $i\in\left\{ 3,\ldots,n-1\right\} $
such that $\phi_{i}\dots\phi_{2}'=\phi_{i}\dots\phi_{2}\phi_{1}$
is an isomorphism. 

If $\phi_{2}$ and $\phi_{1}$ are in general position and either
$\phi_{n}\dots\phi_{2}$ is an isomorphism or $\mathcal{B}(\phi_{n}\dots\phi_{2})=\mathcal{B}(\phi_{2})$. Then in both cases $\mathcal{B}(\phi_{n}\dots\phi_{2}\phi_{1})$ would be equal
to $\mathcal{B}(\phi_{1})$. So $\phi_{n}\dots\phi_{2}$ is not an
isomorphism and its proper base point is different from that of $\phi_{2}$.
By induction hypothesis, there exists an index $j\in\left\{ 3,\ldots,n-1\right\} $
such that $\alpha=\phi_{j}\dots\phi_{2}$ is an isomorphism. Replacing
$\phi_{j+1}$ by the triangular map $\phi_{j+1}'=\phi_{j+1}\alpha$,
we have $f=\phi_{n}\dots\phi_{j+1}'\phi_{1}$ and we are done by
induction. 
\end{proof}

One can think of Proposition \ref{prop:main} as a kind of presentation by generators and relations, the second part saying in particular that there is essentially no relation except the obvious ones given by Corollary \ref{cor:comp}. 
However, even if $(S,B_S) = (S',B_{S'})$ and $f$ is the birational map induced by an automorphism of $V=S\smallsetminus B_S$, in general the triangular maps $\phi_i$ are not birational transformations between isomorphic smooth completions of $V$ (see \S \ref{expl:+4} and \ref{expl:+5} for illustrations of such situations). 
If one insists in having generators that live on a particular model, one possibility is to fix a rule to pass from each possible model to the distinguished one $(S,B_S)$. This is what is done in \cite{GD2}, where the relations are then expressed in terms of (intricate) amalgamated products. \\

Another consequence of  Proposition \ref{prop:main} is that it enables to obtain \emph{ normal forms} for automorphisms of quasi-projective surfaces admitting a smooth completion.
In the following result, and in the rest of the paper, we shall use the notation 
$f^{\psi}$ to denote the conjugate $\psi f\psi^{-1}$. 

\begin{cor}\label{cor:nform}
Let $f\colon (S,B_S)\dashrightarrow (S,B_S)$ be a birational self-map of a smooth completion. 
Then there exists a birational map of smooth completions $\psi\colon(S,B_S)\dashrightarrow (S',B_{S'})$ such that the conjugate $f^{\psi}$ has one of the following properties :

 a) $f^{\psi}$ is a biregular automorphism of the pair $(S',B_{S'})$, 

 b) $f^{\psi}$ is a triangular self-map of $(S',B_{S'})$ with the pair $f^{\psi},f^{\psi}$ in special position,

 c) The pair  $f^{\psi}$, $f^{\psi}$ is in general position.
\end{cor}
\begin{proof}
Suppose $f$ is not biregular, and consider a minimal factorization $f = \phi_n \dots \phi_1$ into triangular maps given by Proposition \ref{prop:main}. 
If the pair $f,f$ is in special position and $f$ is not triangular (that is, $n \ge 2$), then we consider the conjugate $f_{n-1}=\phi_n^{-1}f\phi_n:(S_{n-1},B_{S_{n-1}})\dashrightarrow (S_{n-1},B_{S_{n-1}})$. 
By hypothesis, $\phi_1$ and $\phi_n$ are in special position and so, by Corollary \ref{cor:comp}, $\phi_1 \phi_n$ is either an isomorphism or a triangular map. Thus $f_{n-1}=\phi_{n-1} \dots \phi_2 (\phi_1 \phi_n)$ has length at most $n-1$ and we are done by induction. 
\end{proof}

The existence of normal forms up to conjugacy for automorphisms of $\mathbb{A}^2$ was first noticed by Friedland and Milnor \cite{FM} as a consequence of Jung's Theorem. 
This was the starting point for an exhaustive study of the possible dynamical behavior of these  automorphisms (see \cite{BS8} and references therein). 
In particular, if $g=f^{\psi}$ satisfies Property c) in the conclusion of Corollary \ref{cor:nform} and  has length $n$ then its iterates $g^k$, $k \in \Z$, have length $n|k|$. 
Such a map is thus similar to a composition of generalized H\'enon maps and so one can expect that these maps will always present a chaotic dynamical behavior. 
On the other hand, any finite automorphism of $V$, any one-parameter flow $f_t$ of automorphisms of $V$, or more generally every automorphism contained in an algebraic subgroup of $\Aut(V)$ (see Proposition \ref{prop:algG} below) always corresponds to Case a) or b) in Corollary \ref{cor:nform}.

\subsubsection{Tame automorphisms} \label{sub:tame}

Given any smooth completion $(S,B_{S})$ of $V$, we denote by $\Aut(S,B_{S})$ the group of automorphisms of the pair $(S,B_{S})$. 
Note that since  $B_{S}$ is the support of an ample divisor on $S$, $\Aut(S,B_{S})$ is an algebraic group. 

For every point $p\in B_{S}$, Corollary \ref{cor:comp} implies that the set of triangular self-maps $\phi\colon (S,B_{S})\dashrightarrow (S,B_{S})$ with $\bs(\phi)=\bs(\phi^{-1})=p$ and automorphisms of the pair $(S,B_{S})$ fixing $p$ is a group, which we shall denote by $\tr(S,B_{S},p)$. 
By Lemma \ref{lem:TrCarac}, the latter coincides with the subgroup of automorphisms of $V=S\smallsetminus B_{S}$ preserving the quasi-projective $\mathbb{A}^1$-fibration ${\rho}_{p}:V\rightarrow \mathbb{A}^1$ induced by the rational pencil $\mathcal{P}_{p}$.  
The groups $\tr(S,B_{S},p)$ are not algebraic, but they are countable increasing unions of algebraic subgroups. More precisely, see \cite[Lemma 2.2.3]{BD}, there exists a birational map $\tau:S\dashrightarrow \mathbb{A}^2=\mathrm{Spec}(\mathbb{C}[x,y])$ restricting to a morphism on $V$ such that $\tr(S,B_{S},p)$ is isomorphic to the subgroup of $\Aut(\mathbb{A}^2)$ consisting of automorphisms of the form $(x,y)\mapsto (ax+b,cy+P(x))$, where $a,c\in \mathbb{C}^*$, $P(x)\in \mathbb{C}[x]$, preserving the points blown-up by $\tau$, including infinitely near ones. For every $d\geq 0$, the set of all automorphisms $(x,y)\mapsto (ax+b,cy+P(x))$ of $\mathbb{A}^2$ with $P(x)$ of degree $\leq d$ is an algebraic group and those preserving the points blown-up by $\tau$ form of closed subgroup of it, whence an algebraic group. 

\begin{definition} \label{def:tame}
Let $V$ be a quasi-projective surface admitting a smooth completion $(S,B_S)$, let $\M =\{(\psi_\alpha,p_\alpha)\}_{\alpha \in A}$ be a nonempty collection of pairs consisting for each $\alpha\in A$ of a birational map of smooth completions $\psi_\alpha \colon(S,B_S)\dashrightarrow (S_\alpha,B_{S_\alpha})$ and a point $p_\alpha\in B_{S_\alpha}$. 
An automorphism of $V$ considered as a birational self-map $f$ of $(S,B_S)$ is called:   

a) $\M$-\emph{affine} (short for \emph{affine relatively to the models in $\M$}) if there exists $\alpha$ such that $f^{\psi_{\alpha}}$ is an element of $\Aut(S_\alpha, B_{S_\alpha})$;

b) $\M$-\emph{Jonqui\`eres} if there exists $\alpha$ such that $f^{\psi_{\alpha}}$ is an element of $\tr(S_\alpha,B_{S_\alpha}, p_\alpha)$.

\noindent We denote by $\M\jun(V)$ the subgroup of $\Aut(V)$ generated by $\M$-affine and $\M$-Jonqui\`eres automorphisms. We call it the group of $\M$-\emph{tame automorphisms} of $V$. 
\end{definition}
 
This notion of tameness depends a priori on the choice of the collection $\M$. However by taking the family $\M_{\mathrm{can}}$ consisting of all pairs $(\Psi,p')$ where $\Psi\colon(S,B_S)\dashrightarrow (S',B_{S'})$ is a birational map of smooth completions and $p'$ is a point of $B_{S'}$, we obtain a canonical intrinsic notion of tameness. 

An automorphism of $V$ with associated birational self-map $f$ of $(S,B_S)$ is said to be \emph{generalized affine} (resp. \emph{generalized Jonqui\`eres}) if $f$ is $\M_{\mathrm{can}}$-affine (resp. $\M_{\mathrm{can}}$-Jonqui\`eres). 
We denote 
$$\GTA(V)=\M_{\mathrm{can}}\jun(V)$$ 
the subgroup of $\Aut(V)$ generated by generalized Jonqui\`eres and generalized affine automorphisms.
Its elements will be called \emph{generalized tame} automorphisms of $V$.

In other words, $\GTA(V)$ is generated by automorphisms of $V$ which either preserve a quasi-projective $\mathbb{A}^1$-fibration $V\rightarrow \mathbb{A}^1$ induced by a pencil of the form $\mathcal{P}_p$ on a suitable smooth completion of $V$ or extend to biregular automorphisms of suitable smooth completions of $V$. 
In fact, since  for an element $f\in \Aut(S,B_{S})$ the induced action of $f$ on $B_{S}\simeq \mathbb{P}^1$ always has a fixed point $p$, it follows that  every generalized affine automorphism is also generalized Jonqui\`eres. 
So $\GTA(V)$ coincides with the normal subgroup of $\Aut(V)$ generated by automorphisms preserving an $\mathbb{A}^1$-fibration $V\rightarrow \mathbb{A}^1$ as above.

\begin{prop}\label{prop:algG}
Let $V$ be a quasi-projective surface admitting a smooth completion.
Then for every algebraic subgroup $G$ of $\Aut(V)$, there exists a smooth completion $(S,B_S)$ of $V$ such that $G$ is a subgroup of $\Aut(S,B_{S})$ or of $\tr(S,B_{S},p)$.  
In particular, every algebraic subgroup of $\Aut(V)$ consists of generalized tame automorphisms of $V$.
\end{prop}

\begin{proof}
Let $(S',B_{S'})$ be a smooth completion of $V$. By Sumihiro equivariant completion theorem \cite{Su}, there exists a smooth projective surface $Z$  on which $G$ acts biregularly and a $G$-equivariant open embedding $V\hookrightarrow Z$. 
The induced birational map $\tau\colon Z\dashrightarrow S'$ has finitely many base points, including infinitely near ones and similarly for its inverse. It follows that the number of base points of an element $g$ of $G$ considered as a birational self-map $g\colon(S',B_{S'})\dashrightarrow (S',B_{S'})$ is bounded by the sum of the number of base points of $\tau$ and its inverse.  
This implies in turn that there exists a minimal integer $M(S') \ge 0$ such that the length of any such $g$ is at most $M(S')$ (in the sense
of Definition \ref{def:len}). The bound $M(S')$ depends on the particular completion we choose to realize the birational action of $G$.
Now we choose a smooth completion $(S, B_{S})$ of $V$ such that $M(S)$ is minimal.  
If $M(S)=0$ the birational self-map $g\colon(S,B_{S})\dashrightarrow(S,B_{S})$
associated to every element of $G$ is an automorphism of pairs and hence
$G\subset\Aut(S,B_{S})$.

Assume now that $M(S) \ge 1$, and let $g\colon(S,B_{S})\dashrightarrow(S,B_{S})$ be the birational self-map associated to an element of $G$ which realizes the bound $M(S)$, that is $g = \psi_M \circ \dots \circ \psi_1$ is a composition of $M = M(S)$ triangular maps. 
Let $h\colon(S,B_{S})\dashrightarrow(S,B_{S})$ be the birational self-map associated to another element of $G$.

Suppose first that $\ell(h) \ge 1$ and let $h = \phi_m \circ \dots \circ \phi_1$, where $m=\ell(h)$, be a minimal decomposition of $h$ into triangular maps. 
If $h^{-1}$ and $g$ were in general position then, by Proposition \ref{prop:main}  $\psi_M\cdots \psi_1\phi_1^{-1}\cdots \phi_m^{-1}$ would be a minimal decomposition of $g\circ h^{-1}$, and we would have $\ell(g\circ h^{-1}) = \ell(g) + \ell(h) > M(S)$ in contradiction with the definition of $M(S)$.
So $\bs(\phi_1) = \bs(h) = \bs(g)= \bs(\psi_1)$, and $\psi_1 \circ \phi_1^{-1}$ is either triangular or biregular.
If $\psi_1 \circ \phi_1^{-1}$ is triangular, then by Corollary \ref{cor:comp} the pairs of maps $(\psi_1 \circ \phi_1^{-1}),\psi_2$ and $\phi_2^{-1}, (\psi_1 \circ \phi_1^{-1})$ are both in general position, and  we have $\ell(g \circ h^{-1}) = M + m -1$, so $m = 1$.
This implies that if $\ell(h) \ge 2$ then $\psi_1 \circ \phi_1^{-1}$ is biregular, and applying the same reasoning to $h^{-1}$ instead of $h$ we also get that $\psi_1 \circ \phi_m$ is biregular in this case.
Finally observe that in the case $\ell(h) = 1$ we have $\ell(\psi_1 h \psi_1^{-1}) = \ell(\psi_1 h^{-1} \psi_1^{-1}) \le 1$: indeed either $\psi_1 \circ \phi_1^{-1}$ is biregular and this is clear, or $\psi_1 \circ \phi_1^{-1}$ is triangular with base point equal to  $\bs(\phi_1^{-1}) = \bs(\psi_1)$ and so $\psi_1^{-1}$ and $\psi_1 \circ \phi_1^{-1}$ are in special position.

Consider now the case $\ell(h) = 0$. 
We claim that $h$ fixes $\bs(g) = \bs(\psi_1)$: Otherwise $g^{-1}$ and $gh$  would be in general position and we would have $\ell(ghg^{-1})=2\ell(g)>M(S)$, a contradiction. 
It follows that $\psi_1 h\psi^{-1}_1$ is triangular or biregular in this case. 

In conclusion, conjugating the group $G$ by the birational map $\psi_1\colon(S,B_S)\dashrightarrow (S_1,B_{S_1})$, we obtain that $\ell(\psi_1 h \psi_1^{-1}) \le 1$ if $\ell(h)\leq 1$ whereas $\ell(\psi_1 h \psi_1^{-1}) = l(h) -2$ if $\ell(h) \ge 2$. 
So $M(S)=1$ for otherwise we would have $M(S_1)< M(S)$, in contradiction with the minimality of $M(S)$. 
This shows that all elements in $G$ extend to biregular or triangular maps from $S$ to itself. 
The argument above shows that the point $p = \bs(g)$ is fixed by all biregular elements of $G$ and is the proper base point of all triangular elements in $G$, that is, $G \subset \tr(S,B_{S},p)$.
\end{proof}

\subsection{Automorphisms of affine surfaces with smooth completions} 
Here we consider the particular case of affine surfaces $V$  admitting smooth completions $(S,B_S)$. 
By Proposition \ref{prop:models}, every such surface is isomorphic to $\mathbb{P}^2\smallsetminus C$ where $C$ is a line or a smooth conic or to the complement of an ample section $C$ in a Hirzebruch surface $\pi_n\colon\mathbb{F}_n\rightarrow \mathbb{P}^1$. 
As we saw before, the integer $d=B_S^2$ is an invariant of $V$ and in fact the only invariant except in the case $d=4$. 
Indeed, by the Danilov-Gizatullin Isomorphy Theorem \cite{GD2}, the isomorphy type as an abstract affine surface of the complement of an ample section $C$ in a Hirzebruch surface depends neither on the ambient projective surface nor on the choice of the section, but only on its self-intersection. 
The following proposition summarizes some of the properties of the automorphism groups of these surfaces.

\begin{prop} \label{prop:Affautos} For an affine surface $V_d$ admitting a smooth completion $(S,B_S)$ with $B_S^2=d>0$, the following holds :

1) If $d\leq 4$ then every automorphism of $V_d$ is generalized tame, and one has  $\Aut(V_d)=\mathcal M\jun(V_d)$ for a finite family $\mathcal M$ of completions. 
In particular, $\Aut(V_d)$ is generated by countably many algebraic subgroups. 

2) If $d\geq 5$ then $\GTA(V_d)$ is a proper normal subgroup of $\Aut(V_d)$ and it cannot be generated by countably many algebraic subgroups.

\end{prop}

\begin{proof} The fact that for every $d\leq 4$, $\Aut(V_d)=\M\jun(V_d)$ for a natural choice of finitely many smooth completions $\M$ is checked in the examples in Section \ref{sec:4}.
The second part of assertion 1) then follows from the fact groups of the form $\Aut(S,B_S)$ and $\tr(S,B_S,p)$ are generated by countable families of algebraic subgroups (see \S \ref{sub:tame}).

For the second assertion, we first need to prove that if $d\geq 5$ then there exist smooth completions $(S,B_S)$ of $V_d$ with the property that the orbits of the induced action of  $\Aut(S,B_S)$ on $B_S$ are finite. Namely, if $d > 7$ is even then $V_d$ admits a smooth completion $(\mathbb{F}_2,C)$ where $C$ is a section of $\pi_2\colon \mathbb{F}_2\rightarrow \mathbb{P}^1$ with self-intersection $d$ intersecting the exceptional section $C_0$ of $\pi_2$ with self-intersection $-2$ in $(d-2)/2\geq 3$ distinct points. Similarly, if $d\geq 7$ is odd, then there exists a smooth completion $(\mathbb{F}_1,C)$ of $V$ where $C$ is a section of $\pi_1:\mathbb{F}_1\rightarrow \mathbb{P}^{1}$ with self-intersection $d$ intersecting the exceptional section $C_0$ of $\pi_1$ with self-intersection $-1$ in $(d-1)/2\geq 3$ distinct points. Since in each case the set $C\cap C_0$ is necessarily globally preserved by the induced action of the automorphism group of the ambient pair on $C\simeq \mathbb{P}^1$, we conclude that the orbits of this group on $C$ are finite. The remaining two cases $d=5,6$ are treated explicitly in the examples in Section \ref{sec:4}, \S  \ref{expl:+5-models} 1-b) and \S  \ref{expl:+6} respectively. 

From now on, we identify $V=V_{d}$, $d\geq 5$, to $S\setminus B_{S}$ for a fixed smooth completion $(S,B_{S})$ with the property that the orbits of the induced $\Aut(S,B_S)$-action on $B_S$ are finite.

Let us show first that $\GTA(V)$ cannot be generated by a countable family $(G_{i})_{i\in\mathbb{N}}$ of algebraic subgroups. 
By Proposition \ref{prop:algG}, to each $G_i$ is associated a smooth completion $(S_{i},B_{S_{i}})$
of $V$ and a birational map of pairs $\psi_{i}\colon(S,B_{S})\dashrightarrow(S_{i},B_{S_{i}})$,
such that $\psi_{i}G_{i}\psi_{i}^{-1}$ is contained
either in $\Aut(S_{i},B_{S_{i}})$ or in $\mathrm{Tr}(S_{i},B_{S_{i}},p_{i})$
for some point $p_{i}\in B_{S_{i}}$. 
For every $i\in\mathbb{N}$,
we fix a minimal decomposition $\psi_{i}=\phi_{i,n_{i}}\dots\phi_{i,1}\colon(S,B_{S})=(S_{i,0},B_{S_{i,0}})\dashrightarrow(S_{i,n_{i}},B_{S_{i,n_{i}}})=(S_{i},B_{S_{i}})$ of $\psi_{i}$ into triangular maps $\phi_{i,k}\colon(S_{i,k-1},B_{S_{i,k-1}})$ $\dashrightarrow(S_{i,k},B_{S_{i,k}})$.
Let $q_{ij}=\mathcal{B}(\phi_{i,j+1})$ and $r_{ij}=\mathcal{B}(\phi_{i,j}^{-1})$: Observe that $q_{ij}$ and $r_{ij}$ are both points in $B_{S_{i,j}}$.
Then define $\mathcal{C}$ as the subset of $B_S$ consisting of points of the form $\alpha_{i,j}^{-1}(q_{ij})$, $\beta_{i,j}^{-1}(r_{i,j})$ and $\gamma_i^{-1}(p_i)$ for all possible isomorphisms of pairs $\alpha_{i,j}\colon(S,B_S)\rightarrow (S_{i,j},B_{S_{i,j}})$, $\beta_{i,j}\colon(S,B_S)\rightarrow (S_{i,j},B_{S_{i,j}})$ and $\gamma_i\colon(S,B_S)\rightarrow (S_i,B_{S_i})$ respectively. 

It then follows from Corollary \ref{cor:cancel} that the set of possible proper base points of elements of $\Aut(V)$ considered as birational self-maps of $(S,B_S)$ is contained in $\mathcal{C}$.
Note that Proposition \ref{prop:main} b) implies that $\mathcal{C}$ does not depend on the choice of the minimal decompositions of the birational maps $\psi_i$. 
Our choice of $(S,B_S)$ implies that the $\Aut(S,B_S)$-orbit of $\mathcal{C}$ is countable (in fact one could show that $\mathcal C$ is already $\Aut(S,B_S)$-invariant, but this is not necessary for the argument). 
So for every point $p$ in its complement, a strictly birational element in $\tr(S,B_S,p)$ is an element of $\GTA(V)$ which does not belong to the subgroup generated by the $G_i$.

Next, to derive that $\GTA(V)$ is a proper subgroup of $\Aut(V)$, we exploit a more precise version of the Danilov-Gizatullin Isomorphy Theorem (see \cite[\S 3.1]{DubFin}, in particular Lemma 3.2 and the discussion just before) which asserts that if $(S',B_{S'})$ and $(S'',B_{S''})$ are  smooth completions of $V$ and $p'\in B_{S'}$, $p''\in B_{S''}$ are general points, then the  $\mathbb{A}^1$-fibered affine surfaces
 $\rho'\mid_{V'} \colon V'=S'\smallsetminus B_{S'}\rightarrow \mathbb{A}^1$ and $\rho''\mid_{V''}\colon V''=S''\smallsetminus B_{S''}\rightarrow \mathbb{A}^1$ are isomorphic.
More precisely, this holds whenever $\rho'$ and $\rho''$ have reduced fibers, a property which is always satisfied for general $p'$ and $p''$. 
In view of Lemma \ref{lem:TrCarac}, this implies in particular that for every pair of general points $p$ and $p'$ in $B_S$, there exists a triangular map $\theta\colon(S,B_S)\dashrightarrow (S,B_{S})$ with $\bs(\theta)=p$ and $\bs(\theta^{-1})=p'$. 
By virtue of Lemma \ref{TamL1Carac} below, such a map $\theta$ corresponds to an element of $\GTA(V)$ if and only if $\mathcal{B}(\theta)$ and $\mathcal{B}(\theta^{-1})$ belong to a same orbit of the induced action of $\Aut(S,B_S)$ on $B_S$. Since $\Aut(S,B_S)$ acts on $B_S$ with finite orbits, it follows that we can find two general points $p,p'$ of $B_S$ in distinct orbits; a corresponding triangular self-map $\theta$ then induces an element of $\Aut(V)\setminus \GTA(V)$.
\end{proof}

In the proof of the previous theorem, we used the following characterization of generalized tame automorphisms of length $1$: 

\begin{lem} \label{TamL1Carac}
A triangular map $\theta\colon (S,B_{S})\dashrightarrow(S,B_{S})$ of smooth completions is a generalized tame automorphism of $S\setminus B_{S}$ if and only if $\bs(\theta)$
and $\bs(\theta^{-1})$ belong to a same orbit of the 
action of $\Aut(S,B_{S})$ on $B_{S}$. 
\end{lem}

\begin{proof}
Clearly if there exists an automorphism $\alpha \in \Aut(S,B_{S})$ such that $p = \bs(\theta) = \alpha \bs(\theta^{-1})$, then $\alpha \theta \in \mathrm{Tr}(S,B_S,p)$, hence $\theta$ is generalized tame.

Conversely, let $\theta$ be a generalized tame triangular map. For any automorphism $\alpha \in \Aut(S,B_{S})$, $\alpha \theta$ is again generalized tame and triangular hence can be written in the form
$$\alpha \theta=\prod_{i=1}^{n}g_{i}^{\psi_{i}} = \psi_{n}g_{n}\psi_{n}^{-1} \dots \psi_{1}g_{1}\psi_{1}^{-1}$$ 
where for every $i=1,\dots,n$, $\psi_{i} \colon (S_{i},B_{S_{i}}) \dashrightarrow (S,B_{S})$ is a birational map of smooth completions, and  $g_{i}\in\Aut(S_{i},B_{S_{i}})\cup\bigcup_{p_{i}\in B_{S_{i}}}\mathrm{Tr}(S_{i},B_{S_{i}},p_{i})$. 
Among all such factorizations (for all choices of $\alpha$), we choose one with the property that $L=\sum_{i=1}^{n} (\ell(\psi_{i}) + \ell(g_{i}) + \ell(\psi_{i}^{-1})) $
is minimal, which implies that for all $i$ we have $\ell(g_{i}^{\psi_{i}}) = 2\ell(\psi_i) + \ell(g_i)$.
Furthermore among all factorizations realizing this minimum $L$, we pick one with the minimal number $n$ of factors.

Now if a composition of the form $\beta = \prod_{i=i_{0}}^{j_{0}}g_{i}^{\psi_{i}}$,
where $1\leq i_{0}\leq j_{0}\leq n$, were an isomorphism, then for all $k > j_0$ we could write  
$$ g_{k}^{\psi_k}\beta = \beta \beta^{-1}  g_{k}^{\psi_k}\beta
= \beta g_{k}^{\beta^{-1} \psi_k}  $$
So we could shift the $\beta$ all the way to the left, without increasing $L$ since $\ell(\psi_k) = \ell(\beta^{-1} \psi_k)$, and this would contradict the minimality of $n$.
Thus we can assume that no composition of the form $\prod_{i=i_{0}}^{j_{0}}g_{i}^{\psi_{i}}$ is an isomorphism, or in other words that $\ell(\prod_{i=i_{0}}^{j_{0}}g_{i}^{\psi_{i}}) \ge 1$.
Taking $i_0 = j_0$ we obtain in particular that no $g_i^{\psi_i}$ is an isomorphism.  
We deduce from Corollary \ref{cor:concat} that $\bs(g_{i}^{\psi_{i}}) = \bs(\psi_i^{-1})$ if ${\psi_{i}}$ is not an isomorphism, and $\bs(g_{i}^{\psi_{i}}) = \psi_i(\bs(g_i))$ otherwise.
We also observe that in both cases $\bs(g_{i}^{\psi_{i}}) = \bs ((g_{i}^{\psi_{i}})^{-1})$.

Now we check that we are in position to apply Corollary \ref{cor:cancel}.
Since we already know that no composition of the form $\prod_{i=i_{0}}^{j_{0}}g_{i}^{\psi_{i}}$ is an isomorphism, and since $\ell(g_{i}^{\psi_{i}}) = 2\ell(\psi_i) + \ell(g_i)$, it remains to exclude the existence of two indexes $j_0 > i_0$ such that $g_{j_0}^{\psi_{j_0}} = fh$ with $f, h$ in general position and $\beta = h \prod_{i=i_0}^{j_0-1} g_i^{\psi_i}$ an isomorphism. 
In such a case we would have
$$\prod_{i = i_0}^ {j_0} g_i^{\psi_i} = f\beta = h^{-1}\beta\beta^{-1}hf\beta = \left(\prod_{i=i_0}^{j_0-1} g_i^{\psi_i}\right) \beta^{-1}hf\beta.$$
But on one hand $\beta^{-1}hf\beta = (\beta^{-1}h) fh (h^{-1}\beta) = \left(\prod_{i=i_0}^{j_0-1} g_i^{\psi_i}\right)^{-1} g_{j_0}^{\psi_{j_0}} \left(\prod_{i=i_0}^{j_0-1} g_i^{\psi_i}\right)$ is a conjugate of $g_{j_0}$, and on the other hand $h,f$ are in special position. 
So $\ell (\beta^{-1}hf\beta) = \ell(hf) < \ell(fh)$, in contradiction with the minimality of $L$.

Therefore,  it follows from Corollary \ref{cor:cancel} and the observation made on the proper base point of $g_i^{\psi_i}$ that for all $i_0 = 2,\dots, n$:
$$\bs \left(g_{i_0}^{\psi_{i_0}}\right)
=\bs \left(\prod_{i=i_0}^{n}g_{i}^{\psi_{i}}\right)  
\quad \text{ and } \quad
\bs \left( \left(\prod_{i=1}^{i_0-1}g_{i}^{\psi_{i}}\right)^{-1}\right) 
= \bs \left(g_{i_0-1}^{\psi_{i_0-1}}\right).$$ 

Furthermore if  $\prod_{i=1}^{i_{0}-1}g_{i}^{\psi_{i}}$ and $\prod_{i=i_{0}}^{n}g_{i}^{\psi_{i}}$  were in general position for some index $i_0$, then by Corollary \ref{cor:concat} we would have 
$$\ell \left(\prod_{i=1}^{n}g_{i}^{\psi_{i}}\right) = \ell \left(\prod_{i=i_{0}+1}^{n}g_{i}^{\psi_{i}}\right) + \ell \left(\prod_{i=1}^{i_{0}}g_{i}^{\psi_{i}}\right)  \ge 2,$$ 
in contradiction with the fact that $\alpha \theta$ is triangular. So all such compositions are in special position, and we obtain $\bs \left(g_{i_0}^{\psi_{i_0}}\right) = \bs \left(g_{i_0-1}^{\psi_{i_0-1}}\right).$
Finally $\bs \left(g_n^{\psi_n}\right) = \bs \left(g_1^{\psi_1}\right)$, and  we have
$$\bs(\theta) = \bs (g_1^{\psi_1}) = \bs (g_n^{\psi_n}) = \bs (\theta^{-1}\alpha^{-1}) = \alpha \bs(\theta^{-1}). \qedhere $$ 
\end{proof}

\section{Examples}\label{sec:4}

Here we illustrate our algorithm by describing the automorphism groups of affine varieties admitting a smooth completion with boundary $B_S$ of self-intersection at most $6$. 
We first check that we recover the well-known structure of the automorphism group of $\mathbb{A}^2$. 
Then we consider the case of the affine quadric surface $\CP^1\times \CP^1\smallsetminus \Delta$ started in Section \ref{sec:2} for which we recover in particular the description of its automorphism group given for instance in \cite{Lamyquad}. 
We briefly discuss the case of the complement of a smooth conic in $\mathbb{P}^2$ which is similar. 
As a next step, we describe the situation for affine surfaces admitting smooth completions with boundaries of self-intersection $3$ and $4$ for which two new phenomena arise successively: the existence of non isomorphic $\mathbb{A}^1$-fibrations associated to rational pencils $\mathcal{P}_p$ with proper base point on the boundary, and the existence of non isomorphic smooth completions of a given affine surface. 
We observe that all these examples share the common property that their automorphisms are tame. 
Finally, we consider the more subtle situations of affine surfaces admitting  smooth completions with boundary of self-intersection $5$ and $6$ for which we establish the existence of non-tame automorphisms.

\subsection{Automorphisms of $\mathbb{A}^2$, $\mathbb{P}^2\smallsetminus \{\textrm {smooth conic}\}$ and of the smooth affine quadric surface}

\subsubsection{The affine plane} \label{expl:c2}

Here we derive Jung's Theorem from the description of the triangular maps  which appear in the factorization of an automorphism of $\mathbb{A}^2$ (see also \cite{LamyJung} and \cite{Mat}, which contain proofs of Jung's Theorem derived from the philosophy of the (log) Sarkisov program). 
We let 
$$\mathbb{A}^2=\mathrm{Spec}(\mathbb{C}[u,v])=\mathbb{P}^2\smallsetminus L_0=\mathrm{Proj}(\mathbb{C}[x,y,z])\smallsetminus \{z=0\}$$ with $(u,v)=(x/z,y/z)$, and we define the affine and Jonqui\`eres automorphisms with respect to this unique completion, with the choice of $p_\infty = [1:0:0]$ (see Definition \ref{def:tame}). 
The restriction to $\mathbb{A}^2$ of the rational pencil $\rho\colon\mathbb{P}^2\dashrightarrow \mathbb{P}^{1}$ generated by $L_0$ and the line ${y=0}$ coincides with the second projection $\mathrm{pr}_v\colon \mathbb{A}^2\rightarrow \mathbb{A}^1$. 
Since the pairs $(\mathbb{P}^2,L)$ where $L$ is a line are the only smooth completions of $\mathbb{A}^2$, our algorithm leads to a factorization of an arbitrary polynomial automorphism $f$ of $\mathbb{A}^2$ into a finite sequence of triangular birational maps $\phi_i\colon (\mathbb{P}^2,L_{i-1})\dashrightarrow (\mathbb{P}^2,L_{i})$, where the $L_i$ are lines. 
These maps are obtained as sequences of $2k_i-2$, $k_i\geq 2$, elementary links as in Theorem \ref{th:main} of the form 
 $$ (\cpd,L_{i-1}) \dashrightarrow (\cpd(2),B_1) \dashrightarrow \dots \dashrightarrow (\cpd(k_i),B_{k_i-1}) \dashrightarrow \dots \dashrightarrow (\cpd(2),B_{2k_i-3}) \dashrightarrow (\cpd,L_{i})$$
where  $\cpd(d)$ is the weighted projective plane $\cpd(d,1,1)$, obtained from the Hirzebruch surface $\pi_d\colon\mathbb{F}_d\rightarrow \mathbb{P}^1$ by contracting the section with self-intersection $-d$, and each intermediate boundary is the image by the contraction of a fiber of $\pi_d$. 
Since $\Aut(\mathbb{P}^2)$ acts transitively on pairs consisting of a line and a point of it, we may associate to each $\phi_i$ two isomorphisms of pairs 
$\alpha_i \colon(\mathbb{P}^2,L_0)\rightarrow (\mathbb{P}^2,L_{i-1})$ and
$\beta_i \colon(\mathbb{P}^2,L_0)\rightarrow (\mathbb{P}^2,L_i)$, which map $p_\infty=[1:0:0]$ respectively onto the proper base points of $\phi_i$ and $\phi_i^{-1}$.
Then the induced birational map $\beta_i^{-1}\phi_i\alpha_{i}$ restricts on $\mathbb{A}^2$ to an automorphism commuting with the second projection. 
Thus each $\beta_i^{-1}\phi_i\alpha_{i}$ is a triangular automorphism of $\mathbb{A}^2$ in the usual sense (but in this paper, we prefer to use the terminology ``Jonqui\`eres'', since we reserve ``triangular'' for a more general notion). 
Thus every automorphism of $\mathbb{A}^2$ is $\mathcal M$-tame with respect to the family $\M=\{(\mathrm{id}_{\mathbb{P}^2},p_\infty)\}$ and via Proposition \ref{prop:main}, we recover the classical description of the automorphism group of $\mathbb{A}^2$ as the free product of its affine and Jonqui\`eres subgroups amalgamated along their intersection. 
We also recover via Proposition \ref{prop:algG} another classical fact: any algebraic subgroup of $\Aut(\mathbb{A}^2)$ is conjugated to a subgroup of affine or Jonqui\`eres automorphisms.

\subsubsection{The smooth affine quadric surface} \label{ex:+2}
The structure of the automorphism group of the smooth affine quadric surface $Q=\mathrm{Spec}(\mathbb{C}[x,y,z]/(xz-y(y+1))$ is similar to that of the affine plane. Via the open embedding 
$$Q\hookrightarrow  \mathbb{F}_0=\mathbb{P}^1\times \mathbb{P}^1=\mathrm{Proj}(\mathbb{C}[u_0:u_1])\times\mathrm{Proj}(\mathbb{C}[v_0:v_1]),\;(x,y,z)\mapsto ([x:y],[x:y+1])=([y+1:z],[y:z]),$$ we identify $Q$ with the complement of the diagonal $D_0=\{u_0v_1-u_1v_0=0\}$. 
The rational pencil on $\mathbb{F}_0$ generated by $D_0$ and the union of the two rules $\{u_1=0\}\cup\{v_1=0\}$ has a unique proper base point $p=([1:0],[1:0])$ and the restriction to $Q$ of the corresponding rational map $\rho\colon\mathbb{F}_0\dashrightarrow \mathbb{P}^1$ coincides with the $\mathbb{A}^1$-fibration $\mathrm{pr_z}\colon S\rightarrow \mathbb{A}^1$. 
The minimal resolution $X\rightarrow \mathbb{F}_0$ is obtained by blowing-up two times the point $p$ with successive exceptional divisors $F$ and $C_0$. 
The surface $X$ then dominates the Hirzebruch surface $\pi_1\colon\mathbb{F}_1\rightarrow \mathbb{P}^1$ via the contraction of the strict transforms of $\{u_1=0\}$ and $\{v_1=0\}$. 
Since $\mathrm{Pic}(Q)\simeq\mathbb{Z}$, it follows from Proposition \ref{prop:models} that the pairs $(\mathbb{F}_0,D)$ where $D$ is a smooth curve of type $(1,1)$ are the only possible smooth completions of $Q$. 
Proposition \ref{prop:main} and the description of the resolution of triangular maps given in \S \ref{Tr-disc} lead to a decomposition of every automorphism of $Q$ into a sequence of triangular maps 
$$\phi_i\colon(\mathbb{F}_0,D_{i-1})\dashrightarrow (\hat{\mathbb{P}}^2(2),B_1) \dashrightarrow \dots \dashrightarrow (\hat{\mathbb{P}}^2(k_i),B_{k_i-1})\dashrightarrow \dots \dashrightarrow(\hat{\mathbb{P}}^2(2),B_{2k_i-3})\dashrightarrow (\mathbb{F}_0,D_i)$$ 
where the projective surface $\hat{\mathbb{P}}^2(d)$ is  obtained  from the Hirzebruch surface  $\pi_d\colon\F_{d}\rightarrow \mathbb{P}^1$, with negative section  $C_0$,  by first blowing-up two distinct points in a fiber $F\smallsetminus C_0$  of  $\F_{d}\rightarrow\cpo$ and then contracting  successively the strict transforms of $F$ and $C_0$. 
The intermediate boundaries are the images by these contraction of a fiber $F'$ of $\mathbb{F}_d$ distinct from $F$. Remark that $\hat{\mathbb{P}}^2(d)$  dominates the weighted projective plane $\cpd(d)$ via a single divisorial contraction, hence the  notation.

Since $\Aut(\mathbb{F}_0)$ acts transitively on the set of pairs consisting of a smooth curve of type $(1,1)$ and a point on it, we may associate to each $\phi_i$ two isomorphisms of pairs 
$\alpha_i \colon(\mathbb{F}_0,D_0)\rightarrow (\mathbb{F}_0,D_{i-1})$ and
$\beta_i \colon(\mathbb{F}_0,D_0)\rightarrow (\mathbb{F}_0,D_i)$, which map $p_\infty=([1:0],[1:0])$ respectively onto the proper base points of $\phi_i$ and $\phi_i^{-1}$.
Thus the induced birational map $\beta_i^{-1}\phi_i\alpha_{i}$ restricts on $Q$ to an automorphism commuting with the $\mathbb{A}^1$-fibration $\mathrm{pr}_z$.
Such automorphisms come as the lifts via the morphism ${\rm p}_{z,y}:Q\rightarrow \mathbb{A}^2$ of Jonqui\`eres automorphisms of $\mathbb{A}^2$ of the form $(z,y)\mapsto (az,y+zP(z))$, or the form $(z,y)\mapsto (az,-(y+1)+zP(z))$ where $a\in \C^*$ and $P(z)$ is a polynomial. Every automorphism of the second family  is obtained from one of the first family by composing with the affine involution $(z,y)\mapsto (z,-(y+1))$ of $\mathbb{A}^2$ which lifts to the involution $Q$ induced by the "symmetry" with respect to the diagonal $D_0$ in $\mathbb{P}^1\times \mathbb{P}^1$. We recover in this way the presentation given in \cite{Lamyquad} of $\Aut(V)$ as the amalgamated product of $\mathrm{Aff}(Q)=\Aut(\mathbb{F}_0,\Delta)\mid_Q$ and $\Aut(Q,\mathrm{pr}_z)=\tr(\mathbb{F}_0,\Delta,p_\infty)\mid_Q$ over their intersection. 
In particular, similarly as in the case of $\mathbb{A}^2$, we have $\Aut(Q)=\M\jun(Q)$, with respect to the family $\M=\{(\mathrm{id}_{\mathbb{F}_0},p_\infty)\}$. 
Proposition \ref{prop:algG} says in turn that every algebraic subgroup of $\Aut(V)$ is conjugated to a subgroup of $\mathrm{Aff}(Q)$ or $\Aut(Q,\mathrm{pr}_z)$. For instance, every algebraic action of the additive group $\mathbb{G}_a$ on $Q$ is conjugated to an action preserving the fibration $\mathrm{pr}_z$.

\subsubsection{The complement of a smooth conic} \label{expl:p2c} 

Since the automorphism group of $\mathbb{P}^2$ acts transitively on the set of pairs consisting of a smooth conic and a point of it, every smooth pair $(\mathbb{P}^2,C)$ where $C$ is a smooth conic is isomorphic to $(\mathbb{P}^2,C_0)$ where $C_0=\{yz-x^2=0\}$, and every rational pencil $\mathcal{P}_p$ associated to a point on $C$ is conjugated to $\mathcal{P}_{p_0}$ generated by $C_0$ and $2L_0$ where $L_0$ denotes the tangent to $C_0$ at the point $p_0=[0:0:1]$. 
The corresponding $\mathbb{A}^1$-fibration $q=\rho_{p_0}\mid_V:V=\mathbb{P}^2\smallsetminus C_0\rightarrow \mathbb{A}^1$ has a unique degenerate fiber $L_0\cap V$, of multiplicity $2$. The automorphism group of $V$ is then the amalgamated product of $\mathrm{Aff}(V)=\Aut(\mathbb{P}^2,C_0)\mid_V$ and $\Aut(V,q)=\tr(\mathbb{P}^2,C_0,p_0)\mid_V$ over their intersection. 
Again, we have $\Aut(V)=\M\jun(V)$, for $\mathcal M$ a collection consisting of a unique model.   
The interested reader can find more details in \cite[\S 1.1]{HDR}.

\subsection{Complement of a section with self-intersection $3$ in $\mathbb{F}_1$}\label{expl:+3}

By Proposition \ref{prop:models}, a smooth completion of $(S,B_S)$ of an affine surface $V$ with $B_S^2 \neq 1,4$ is of the form $(\mathbb{F}_n,C)$ where $C$ is an ample section. 
Since $B_S^2$ is an invariant of $V$, we see that if $B_S^2=3$ the only smooth completions of $V$ are pairs $(\mathbb{F}_1,C)$ where $C$ is a section of self-intersection $3$. 
If we identify  $\pi_1\colon \mathbb{F}_1\rightarrow \mathbb{P}^1$ with the blow-up $\sigma$ of $\mathbb{P}^2=\mathrm{Proj}(\mathbb{C}[x,y,z])$ at the point $q=[0:1:0]$ with exceptional divisor $C_0$ and denote by $F_\infty$ the strict transform of the line $\{z=0\}\subset \mathbb{P}^2$, then every such section $C$ is the strict transform of a smooth conic in $\mathbb{P}^2$ passing through $q$. 
The automorphism group of $\mathbb{F}_1$ acts transitively on such sections and so, every smooth completion of $V$ is isomorphic to $(\mathbb{F}_1,D_0)$ where $D_0$ denotes the strict transform of the conic $\{yz-x^2=0\}\subset \mathbb{P}^2$  tangent to the line $\{z=0\}$ at $q$. 
The automorphism group of the pair $(\mathbb{F}_1,D_0)$ acts on $D_0$ with two orbits: the point $p_\infty=D_0\cap C_0=D_0\cap F_\infty$ and its complement. This implies in turn that every rational pencil $\mathcal{P}_{p}$ associated to a point on $D_0$ is conjugated either to $\mathcal{P}_{p_\infty}$ or to $\mathcal{P}_{p_0}$ where $p_0=\sigma^{-1}([0:0:1])$. 
Both of these pencils have a unique singular member consisting of the divisor $C_0+2F_\infty$ in the first case and $L+F_0$ where $L$ and $F_0$ are the respective strict transforms of the lines $\{y=0\}$ and $\{x=0\}$ of $\mathbb{P}^2$ in the second case. 
The induced $\mathbb{A}^1$-fibrations on $V\simeq \mathbb{F}_1\smallsetminus D_0$ are not isomorphic: the one induced by $\mathcal{P}_{p_\infty}$ has a unique degenerate scheme theoretic fiber which consists of the union of two affine lines $C_0\cap V$ and $F_\infty\cap V$ where $F_\infty\cap V$ occurs with multiplicity $2$, while the one induced by $\mathcal{P}_{p_0}$ has a unique degenerate fiber consisting of two reduced affine lines $L\cap V$ and $F_0\cap V$. 
In particular we see from Lemma \ref{lem:TrCarac} that any triangular map $V \to V$ is either in $\tr(\mathbb{F}_1,D_0,p_\infty)$, or up to left-right composition by automorphisms of $(\mathbb{F}_1,D_0)$, in $\tr(\mathbb{F}_1,D_0,p_0)$.

Now given an automorphism of $V$ considered as a birational self-map $f$ of $(\mathbb{F}_1,D_0)$ with decomposition 
$$f=\phi_n\circ\dots\circ\phi_1\colon(\mathbb{F}_1,D_0)=(S_0,B_{S_0})\dashrightarrow \dots \dashrightarrow (S_i,B_{S_i})\dashrightarrow \dots \dashrightarrow(S_n,B_{S_n})=(\mathbb{F}_1,D_0)$$ 
into triangular maps, we can find isomorphisms $\alpha_{i}\colon (\mathbb{F}_1,D_0)\rightarrow (S_{i-1},B_{S_{i-1}})$, $\beta_i\colon (\mathbb{F}_1,D_0)\rightarrow (S_i,B_{S_i})$, $i=1,\ldots, n$ such that for every $i=1,\ldots,n$, $\psi_i=\beta_{i}^{-1}\phi_i\alpha_{i}$ is an element of either $\tr(\mathbb{F}_1,D_0,p_0)$ or $\tr(\mathbb{F}_1,D_0,p_\infty)$. 
Writing $f$ as $$f=\beta_n\psi_n(\alpha_{n}^{-1}\beta_{n-1})\dots\beta_2\psi_2(\alpha_2^{-1}\beta_1)\psi_1\alpha_1^{-1}$$
where $\alpha_1^{-1},\beta_n$ and $\alpha_i^{-1}\beta_{i-1}$, $i=2,\ldots,n$ are elements of $\Aut(\mathbb{F}_1,D_0)$, we conclude that  $\Aut(V)$ is generated by  $\Aut(\mathbb{F}_1,D_0)$, $\tr(\mathbb{F}_1,D_0,p_0)$ and $\tr(\mathbb{F}_1,D_0,p_\infty)$.
In other words $\Aut(V) = \M\jun(V)$ for the family $\mathcal M = \left\lbrace (\mathrm{id}_{\F_1}, p_0),(\mathrm{id}_{\F_1}, p_\infty)  \right\rbrace $.
Remark that in this case we cannot use a single model anymore.

\subsection{Complement of a section with self-intersection $4$ in $\F_0$ }\label{expl:+4}

Here we consider the case of an affine surface $V$ admitting
a smooth completion by a smooth rational curve with self-intersection
$4$, and which is not isomorphic to the complement of a conic in $\cpd$.
According to Proposition \ref{prop:models}, the corresponding pairs $(S,B_S)$ are either $(\mathbb{F}_{0},B)$ where
$B$ is an arbitrary smooth rational curve with self-intersection
$4$ (which is automatically a section with respect to one of the two rulings) or $\left(\mathbb{F}_{2},B\right)$ where $B$ is a section of the $\mathbb{P}^{1}$-bundle structure $\pi_2\colon \mathbb{F}_{2}\rightarrow\mathbb{P}^{1}$. 
First we review these smooth  completions $(S,B_S)$ with a particular emphasis on the rational pencils $\mathcal{P}_p$ related with all possible triangular elementary maps that can occur in the factorization given by Proposition \ref{prop:main}. 
Given such a pencil $\mathcal{P}_p$, we let $\sigma:\hat{S}\rightarrow S$ be the minimal resolution of the corresponding rational map $\rho_p\colon S\dashrightarrow \mathbb{P}^1$. 
The last exceptional divisor extracted by $\sigma$ is a section $C$ of the induced $\mathbb{P}^1$-fibration $\rho_{p}\sigma:\hat{S}\rightarrow \mathbb{P}^1$ and one can prove (see \cite{BD}, Section 2) that $\hat{S}$ dominates $\pi_1:\mathbb{F}_1\rightarrow \mathbb{P}^1$ through a uniquely determined sequence of contractions $\tau :\hat{S}\rightarrow \mathbb{F}_1$ in such a way that the general fibers of $\rho_p\sigma$ coincide with that of $\pi_1\tau$ and that the strict transform $\tau_*(C)$ of $C$ coincides with the exceptional section of $\pi_1\colon\mathbb{F}_1\rightarrow \mathbb{P}^1$ with self-intersection $-1$. 
We will use this point of view to give a uniform description of the different pencils involved: see Figure \ref{fig:models}. 

\begin{figure}[ht]
$$\mygraph{
!{<0cm,0cm>;<1cm,0cm>:<0cm,1cm>::}
!{(4.2,4)  }*+{\dessinFun}="F1"
!{(0,0)  }*++{\ModeleI}="MI" 
!{(4.2,-2)  }*++{\ModeleII}="MII"
!{(8.4,0)  }*++{\ModeleIII}="MIII"
"F1"-@{<-}"MI" "F1"-@{<-}"MII" "F1"-@{<-}"MIII" 
}$$
\caption{The three models of pencils with their resolution (the index of the exceptional divisors $E_i$ corresponds to their order of construction coming from $\F_1$). }  
\label{fig:models}
\end{figure}

\subsubsection{The case $\left(\mathbb{F}_{0},B\right)$.} \label{cas:F0} 

With the bi-homogeneous coordinates introduced in \S \ref{ex:+2}, every pair $(\mathbb{F}_{0},B)$ where $B$ is an irreducible curve with $B^2=4$ is isomorphic to $(\mathbb{F}_0,D_0)$ where $D_0=\{u_1^2v_0-u_0^2v_1=0\}$. Letting $C_0=\{v_0=0\}$ and $F_0=\{u_0=0\}$ we have $D_0\sim C_0+2F_0$. 
The automorphism group of the pair $(\mathbb{F}_0,D_0)$ acts on $D_0$ with two orbits : the pair of points $\{p_0',p_\infty'\}=\{([0:1],[0:1]),([1:0],[1:0])\}$ and their complement. 
This implies in turn that there exist only two models of rational pencils ${\mathcal{P}}_{p'}$ up to conjugacy by automorphisms of $(\mathbb{F}_0,D_0)$, say ${\mathcal{P}}_{p_0'}$ and ${\mathcal{P}}_{p_1'}$ where $p_1'=([1:1],[1:1])$. They can be described as follows:

   a) The pencil $\mathcal{P}_{p_1'}$ is generated by $D_0$ and $H+F_1$ where $H=\{(u_0-3u_1)(v_0+3v_1)+8u_1v_1=0\}$ is the unique irreducible curve of type $(1,1)$ intersecting $D_0$ only in $p_1'$, with multiplicity $3$, and  $F_1$ is the fiber of the first ruling over the point $[1:1]$. 
   The restriction of this pencil to $V_0=\mathbb{F}_{0}\smallsetminus D_0$ is an $\mathbb{A}^1$-fibration $V_0\rightarrow\mathbb{A}^{1}$ with a unique degenerate fiber consisting of the disjoint union of two reduced affine lines $H\cap V_0$ and $F_{1}\cap V_0$. See Figure \ref{fig:models} (model I). 

   b) The pencil ${\mathcal{P}}_{p_0'}$ is generated by $D_0$ and $C_{0}+2F_0$ (note that $C_0$ is the tangent to $D_0$ at the point $p_0'$). 
   Its restriction to $V_0$ is an $\mathbb{A}^1$-fibration $V_0\rightarrow\mathbb{A}^{1}$ with a unique degenerate fiber consisting of the disjoint union of a reduced affine line $C_{0}\cap V_0$ and a non reduced one $F_0\cap V_0$, occurring with multiplicity $2$. See Figure \ref{fig:models} (model II).

The fact that the $\mathbb{A}^1$-fibrations associated to the pencils $\mathcal{P}_{p_0'}$ and $\mathcal{P}_{p_1'}$ are not isomorphic implies further that every triangular self-map $\phi:(\mathbb{F}_0,D_0)\dashrightarrow (\mathbb{F}_0,D_0)$ is the product of an element of $\tr(\mathbb{F}_0,D_0,p_0')$ or $\tr(\mathbb{F}_0,D_0,p_1')$ and an element of $\Aut(\mathbb{F}_0,D_0)$.

\subsubsection{The case $\left(\mathbb{F}_{2},B\right)$.} \label{cas:F2}

Letting $C_0$ be the exceptional section of $\pi_2\colon\mathbb{F}_2\rightarrow \mathbb{P}^1$ with self-intersection $-2$, a section $B$ of $\pi_2$ with self-intersection $4$ is linearly equivalent to $C_{0}+3F_\infty$ where $F_\infty$ is a fiber of $\pi_2$. 
In particular $B$ intersects $C_{0}$ transversely in a single point, which we can assume to be $C_{0}\cap F_\infty$. We identify $\mathbb{F}_{2}\smallsetminus\left(C_{0}\cup F_{\infty}\right)$ to $\mathbb{A}^{2}$ with coordinates $x$ and $y$ in such way that the induced ruling on $\mathbb{A}^{2}$ is given by the first projection and that the closures in $\mathbb{F}_2$ of the level sets of $y$ are sections of $\pi_2$ linearly equivalent to $C_0+2F_{\infty}$ (equivalently, the closure of the curve $\{y=0\}$ in $\mathbb{F}_2$ does not intersect $C_0$). 
With this choice, $B$ coincides with the closure of an affine cubic defined by an equation of the form $y=ax^{3}+bx^{2}+cx+d$. 
Since any automorphism of $\mathbb{A}^{2}$ of the form $\left(x,y\right)\mapsto\left(\lambda x+\nu,\mu y+P\left(x\right)\right)$
where $P$ is a polynomial of degree at most $2$, extends to a biregular
automorphism of $\mathbb{F}_{2}$, it follows that every pair $(\mathbb{F}_2,B)$ where $B$ is a section with self-intersection $4$ is isomorphic to $(\mathbb{F}_2,D_2)$ where $D_2$ is the closure in $\mathbb{F}_{2}$ of the affine cubic in $\mathbb{A}^{2}$ with equation $y=x^{3}$. Furthermore, the automorphism group of the pair $(\mathbb{F}_2,D_2)$ acts on $D_2$ with two orbits: the fixed point $p_{\infty}=D_2\cap C_0=D_2\cap F_\infty$ and its complement. 
Again, we have two possible models of rational pencils $\mathcal{P}_p$ up to conjugacy by automorphisms  of $(\mathbb{F}_2,D_2)$, say $\mathcal{P}_{p_0}$ where $p_0=(0,0)\subset \mathbb{A}^2\subset \mathbb{F}_2$ and $\mathcal{P}_{p_{\infty}}$:

a) The pencil $\mathcal{P}_{p_0}$ is generated by $D_2$ and $H+F_0$ where $H\sim C_{0}+2F_\infty$ is the closure 
in $\mathbb{F}_2$ of the affine line $\{ y=0\}\subset\mathbb{A}^2$ which intersects $D_2$ only in $p_0$, with multiplicity $3$, and where $F_0=\pi_2^{-1}([0:1])\subset\mathbb{F}_{2}$. Its restriction to $W_0=\mathbb{F}_{2}\smallsetminus D_2$ is an $\mathbb{A}^1$-fibration $W_0\rightarrow\mathbb{A}^{1}$ with a unique degenerate fiber consisting of the disjoint union of two reduced affine lines $H\cap W_0$ and $F_{0}\cap W_0$. A minimal resolution of this pencil is given in Figure \ref{fig:models} (model I).

b) The pencil $\mathcal{P}_{p_{\infty}}$ is generated by $D_2$ and $C_0+3F_\infty$ (remember that $D_2$ intersects $C_{0}$ transversely in $p_\infty$). Its restriction to $W_0$ is an $\mathbb{A}^1$-fibration $W_0\rightarrow\mathbb{A}^{1}$ with a unique degenerate fiber
consisting of the disjoint union of a reduced affine line $C_{0}\cap W_0$ and a non reduced one $F_\infty\cap W_0$, occurring with multiplicity $3$. See Figure \ref{fig:models} (model III).

\subsubsection{Connecting triangular maps} \label{trconnect}

By the Danilov-Gizatullin Theorem, $(\mathbb{F}_0,D_0)$ and $(\mathbb{F}_2,D_2)$ can arise as smooth  completions of a same affine surface. However, let us briefly explain how to derive this fact directly by constructing appropriate triangular maps $\phi\colon(\mathbb{F}_2,D_2)\dashrightarrow (\mathbb{F}_0,D_0)$. 
In view of Lemma \ref{lem:TrCarac} and of the description of the rational pencils given above, the only possibility is that the proper base points of such a map $\phi$ and its inverse belong respectively to the open orbits of the actions of $\Aut(\mathbb{F}_2,D_2)$ on $D_2$ and of $\Aut(\mathbb{F}_0,D_0)$ on $D_0$. 
Let us construct a particular quadratic triangular map $\Phi_0$ with $\bs(\phi)=p_{0}$ and $\bs(\phi^{-1})=p_1'$ (see Figure \ref{fig:triangF2F0} for the notations). 
\begin{figure}[ht]
$$\mygraph{
!{<0cm,0cm>;<1.1cm,0cm>:<0cm,1cm>::}
!{(-1,+4)  }*++{\dessinTrDeb}="S"
!{(1.8,+8)  }*++{\dessinTrHaut}="Y"
!{(+5,+4)  }*++{\dessinTrFin}="Shat"
!{(0,0)  }*+{\dessinTrFdeux}="F2"
!{(-3,0)  }*+{\dessinTrFun}="F1"
!{(+4,0)  }*+{\dessinTrFzero}="F0"
!{(-3,-2.5)  }*{\dessinTrProj}="P2"
"Y"-@/_2cm/@{->}"S" "Y"-@/^2cm/@{->}"Shat"
"S"-@{->}"F1" "S"-@{->}"F2"
"Shat"-@{->}"F0" "F1"-@{->}"P2"
"S"-@{-->}@/_1cm/^{\hat{\phi}}"Shat"
"F2"-@{-->}@/_1cm/_{\Phi_0}"F0"
}$$
\caption{Quadratic triangular map $\Phi_0 \colon (\F_2,D_2) \dashrightarrow (\F_0,D_0)$.}
\label{fig:triangF2F0}
\end{figure}

Let $\hat{S}\rightarrow\mathbb{F}_{2}$ be the minimal resolution of the base points of the rational pencil $\mathcal{P}_{p_0}$ as in \S \ref{cas:F2}.a) above. The surface $\hat{S}$ can also be obtained from $\mathbb{P}^{2}$ by a sequence of blow-ups with successive exceptional divisors $C$, $E_{1}$,
$E_{2}$, $E_{3}$ and $E_{4}$  in such a way that curves $l$ and $B=D_2$ correspond to the strict transforms of a
pair of lines in $\mathbb{P}^{2}$ intersecting at the center $q$
of the first blow-up.  In this setting, the strict transform of $C_{0}\subset\mathbb{F}_{2}$
in $\hat{S}$ coincides with the strict transform of a certain line
$L$ in $\mathbb{P}^{2}$ intersecting $B$ in a point distinct from
$q$. Let $\hat{\psi}:\hat{S}\dashrightarrow\hat{S}'$ be any fibered
modification of degree $2$ and let $B'$ be the strict 
transform in $\hat{S}'$ of the second exceptional divisor produced.
Then one checks that there exists a unique smooth conic $\Delta$
in $\mathbb{P}^{2}$ tangent to $B$ in $q$ and to $L$ at the point
$L\cap l$ such that its strict transform in $\hat{S}'$
is a $\left(-1\right)$-curve which intersects transversely the strict transforms of $E_{2}$ and $B'$ in general points. 
By successively contracting $E_{4},\ldots,E_{1}$, we arrive at a
new projective surface $S'$ in which the strict transform of $B'$
is a smooth rational curve with self-intersection $4$ and such that
the strict transforms of $\Delta$ and $E_{3}$ are smooth rational
curves with self-intersection $0$, intersecting transversely in a
single point. Thus $S'\simeq\mathbb{F}_{0}$ and  $\hat{\phi}:\hat{S}\dashrightarrow\hat{S}'$
descends to a triangular map $\psi:\left(\mathbb{F}_{2},D_2\right)\dashrightarrow\left(\mathbb{F}_{0},B'\right)$.
Moreover, the proper base point of $\psi^{-1}$ is located at a point
where $B'$ intersects the two rulings transversely. So there exists an isomorphism of pairs $\beta \colon(\mathbb{F}_0,B')\rightarrow (\mathbb{F}_0,D_0)$ such that $\Phi_0=\beta \psi\colon(\mathbb{F}_2,D_2)\dashrightarrow (\mathbb{F}_0,D_0)$ is triangular and maps $\mathcal{P}_{p_0}$ onto ${\mathcal{P}}_{p_1'}$.

\subsubsection{The automorphism group} 

To determine the automorphism group of an affine surface $V$ admitting a smooth completion $(S,B_S)$ with $B_S^2=4$ we can proceed as follows. 
First we may assume up to isomorphism that  $V=\mathbb{F}_2\smallsetminus D_2$. 
Then given an automorphism $\xi$ of $V$ we consider a minimal factorization of the associated birational self-map $f$ of $(\mathbb{F}_2,D_2)$ into triangular maps $$f=\phi_n\circ\dots\circ\phi_1\colon(\mathbb{F}_2,D_2)=(S_0,B_{S_0})\dashrightarrow \dots\dashrightarrow (S_i,B_{S_i})\dashrightarrow \dots \dashrightarrow (S_n,B_{S_n})=(\mathbb{F}_2,D_2)$$ where  each $S_i$ is isomorphic either to $\mathbb{F}_0$ or $\mathbb{F}_2$. 

If the intermediate surfaces $S_{j}$ are not all isomorphic to $\mathbb{F}_{2}$,
then we let $j\in\{1,\ldots,n-1\}$ and $k\in\left\{ j+1,\ldots,n\right\} $
be minimal with the property that $S_{j}\simeq S_{k-1}\simeq\mathbb{F}_{0}$
and $S_{k}\simeq\mathbb{F}_{2}$. 
Replacing if necessary $\phi_{j-1}$,
$\phi_{j}$ and $\phi_{j+1}$ by $\alpha\phi_{j-1}$, $\beta\phi_{j}\alpha^{-1}$
and $\phi_{j+1}\beta^{-1}$ for isomorphisms $\alpha\colon(S_{j-1},B_{S_{j-1}})\rightarrow(\mathbb{F}_{2,}D_{2})$
and $\beta\colon(S_{j},B_{S_{j}})\rightarrow(\mathbb{F}_{0},D_{0})$, we
may assume from the beginning that $(S_{j,-1},B_{j-1})=(\mathbb{F}_{2},D_{2})$
and $(S_{j},B_{j})=(\mathbb{F}_{0},D_{0})$. 
We may assume similarly that $(S_{k-1},B_{S_{k-1}})=(\mathbb{F}_{0},D_{0})$ and $(S_{k},B_{S_{k}})=(\mathbb{F}_{2},D_{2})$.
Now consider the triangular maps $\phi_{j}:(\mathbb{F}_{2},D_{2})\dashrightarrow(\mathbb{F}_{0},D_{0})$
and $\phi_{k}:(\mathbb{F}_{0},D_{0})\dashrightarrow(\mathbb{F}_{2},D_{2})$.
Since the $\mathbb{A}^{1}$-fibrations induced by the pencils $\mathcal{P}_{p_{\infty}}$
on $\mathbb{F}_{2}$ and $\mathcal{P}_{p_{0}'}$ on $\mathbb{F}_{0}$
are not isomorphic and not isomorphic to those associated to points
in $D_{2}\setminus\{p_{\infty}\}$ and $D_{0}\setminus\{p_{0}'\}$
it must be that $\mathcal{B}(\phi_{j})\in D_{2}\setminus\{p_{\infty}\}$
and $\mathcal{B}(\phi_{j}^{-1})\in D_{0}\setminus\{p_{0}'\}$ (see
\S \ref{cas:F0} and \ref{cas:F2}). It follows that there exist automorphisms
$\alpha_{j}\in\Aut(\mathbb{F}_{2},D_{2})$ and $\beta_{j}\in\Aut(\mathbb{F}_{0},D_{0})$
mapping $\mathcal{B}(\phi_{j})$ onto $p_{0}$ and $\mathcal{B}(\phi_{j}^{-1})$
onto $p_{1}'$ respectively. So replacing $\phi_{j-1}$, $\phi_{j}$
and $\phi_{j+1}$ by $\alpha_{j}\phi_{j-1}$, $\beta_{j}\phi_{j}\alpha_{j}^{-1}$
and $\phi_{j+1}\beta_{j}$ respectively, we may assume from the beginning
that $\mathcal{B}(\phi_{j})=p_{0}$ and $\mathcal{B}(\phi_{j}^{-1})=p_{1}'$.
For the same reason, we may assume that $\mathcal{B}(\phi_{k})=p_{1}'$
and $\mathcal{B}(\phi_{k}^{-1})=p_{0}$. 
Strictly speaking, in the case $k = j+1$, we have to insert an automorphism $\alpha \in \Aut(\mathbb{F}_{0},D_{0})$ between $\phi_j$ and $\phi_{j+1}$, which will play the same role as $\phi_{k-1} \dots \phi_{j+1}$ in the sequel.  
Recall that by construction,
the particular triangular map $\Phi_{0}:(\mathbb{F}_{2},D_{2})\dashrightarrow(\mathbb{F}_{0},D_{0})$
constructed in \S \ref{trconnect} has $\mathcal{B}(\Phi_{0})=p_{0}$
and $\mathcal{B}(\Phi_{0}^{-1})=p_{1}'$. It then follows from Corollary
\ref{cor:comp} that $\Phi_{0}^{-1}\phi_{j}:(\mathbb{F}_{2},D_{2})\dashrightarrow(\mathbb{F}_{2},D_{2})$
is an element of $\mathrm{Tr}(\mathbb{F}_{2},D_{2},p_{0})$ while
$\Phi_{0}\phi_{k}:(\mathbb{F}_{0},D_{0})\dashrightarrow(\mathbb{F}_{0},D_{0})$
belongs to $\mathrm{Tr}(F_{0},D_{0},p_{1}')$. Summing up, we can
rewrite $f$ in the form 
\begin{multline*}
f=\phi_{n}\cdots\phi_{k+1}\Phi_{0}^{-1}\left[(\Phi_{0}\phi_{k})\phi_{k-1}\cdots\phi_{j+1}\right]\Phi_{0}\left[(\Phi_{0}^{-1}\phi_{j})\cdots\phi_{1}\right]\\
=f'\Phi_{0}^{-1}\left[(\Phi_{0}\phi_{k})\phi_{k-1}\cdots\phi_{j+1}\right]\Phi_{0}\left[(\Phi_{0}^{-1}\phi_{j})\cdots\phi_{1}\right]
\end{multline*}
where $f':(\mathbb{F}_{2},D_{2})\dashrightarrow(\mathbb{F}_{2},D_{2})$
has length $\ell(f')<\ell(f)$ and where the sequences $(\Phi_{0}\phi_{k})\cdots\phi_{j+1}$
and $(\Phi_{0}^{-1}\phi_{j})\cdots\phi_{1}$ only involve intermediate
surfaces isomorphic to $\mathbb{F}_{0}$ and $\mathbb{F}_{2}$ respectively. 

Now we deduce as in \S \ref{expl:+3} that $(\Phi_{0}^{-1}\phi_{j})\cdots\phi_{1}$
can be written as a sequence of elements in $\Aut(\mathbb{F}_{2},D_{2})$,
$\mathrm{Tr}(\mathbb{F}_{2},D_{2},p_{0})$ and $\mathrm{Tr}(\mathbb{F}_{2},D_{2},p_{\infty})$.
Similarly, $(\Phi_{0}\phi_{k})\cdots\phi_{j+1}$ can be decomposed
into a sequence of elements in $\Aut(\mathbb{F}_{0},D_{0})$,
$\mathrm{Tr}(\mathbb{F}_{0},D_{0},p_{0}')$ and $\mathrm{Tr}(\mathbb{F}_{0},D_{0},p_{1}')$
and so $\Phi_{0}^{-1}\left[(\Phi_{0}\phi_{k})\cdots\phi_{j+1}\right]\Phi_{0}$
can be written as a composition of elements of the conjugates of these
groups by $\Phi_{0}:(\mathbb{F}_{2},D_{2})\dashrightarrow(\mathbb{F}_{0},D_{0})$.
We conclude by induction on the length that $\Aut(V)=\M\mathrm{TA}(V)$
with 
$$\M=\left\{ (\mathrm{id}_{\mathbb{F}_{2}},p_{0}),(\mathrm{id}_{\mathbb{F}_{2}},p_{\infty}),(\Phi_{0},p_{0}'),(\Phi_{0},p_{1}')\right\}.$$

\subsection{Automorphisms of the complement of a section with self-intersection $5$ in $\F_1$ } \label{expl:+5}

While it could seem at first glance similar to the previous ones, this case exhibits a new behavior which is more representative of the general situation: the existence of non-tame automorphisms. 

\subsubsection{Possible models and associated rational pencils} \label{expl:+5-models}

In view of Proposition \ref{prop:models} there exists only two possible types of smooth completions $(S,B_S)$ with $B_S^2=5$ and $S\smallsetminus B_S$ affine: the complements of sections with self-intersection $5$ in either $\pi_1\colon \mathbb{F}_1\rightarrow \mathbb{P}^1$ or $\pi_3\colon \mathbb{F}_3\rightarrow \mathbb{P}^1$.\\ 

1) In the first case, every such section $B$ is linearly equivalent to $C_0+3F$ where $C_0$ is the exceptional section of $\pi_1$ and $F$ a fiber. 
In particular, $B\cdot C_0=2$ and with the notations of \S \ref{expl:+3}, we have two possible pairs up to isomorphisms: first $(\mathbb{F}_1,D_1)$  where $D_1$ is the strict transform of the nodal cubic $C_1=\{x^3-z^3=xyz\}\subset \mathbb{P}^2$ with tangents $\{z=0\}$ and $\{x=0\}$ at $q_0=[0:1:0]$; and second  $(\mathbb{F}_1,D_2)$ where $D_2$ is the strict transform of the cuspidal cubic $C_2=\{x^3=z^2y\}\subset \mathbb{P}^2$ tangent to $\{z=0\}$ at $q_0$. 

a) The automorphism group of $(\mathbb{F}_1,D_2)$ acts on $D_2$ with three orbits : $p_{\infty,2}=C_0\cap F_\infty$, $p_{0,2}=\sigma^{-1}([0:0:1])$ and their complement. 
The pencil $\mathcal{P}_{p_{\infty,2}}$ is generated by $D_2$ and $C_0+3F_\infty$ and it restricts on  $V_{0,2}=\mathbb{F}_1\smallsetminus D_2$ to an $\mathbb{A}^1$-fibration with a unique degenerate fiber consisting of two affine lines $C_0\cap V_{0,2}$ and $F_\infty\cap V_{0,2}$, the second one occurring with multiplicity $3$. The pencil $\mathcal{P}_{p_{0,2}}$ is generated by $D_2$ and $L+2F_0$ where $L\sim C_0+F_\infty$ is the strict transform of the tangent line to $C_2$ at $[0:0:1]$ and $F_0=\pi_1^{-1}\pi_1(p_{0,2})$. 
Its restriction to $V_{0,2}$ is an  $\mathbb{A}^1$-fibration with a unique degenerate fiber consisting of two affine lines $L\cap V_{0,2}$ and $F_0\cap V_{0,2}$, the second one occurring with multiplicity $2$. 
Finally, for every $p\in D_2\smallsetminus(p_{0,2}\cup p_{\infty,2})$, the pencil $\mathcal{P}_{p}$ is generated by $D_2$ and $H+F_p$  where $F_p=\pi_1^{-1}(\pi_1(p))$, and $H\sim C_0+2F_\infty$ is the strict transform of the unique smooth conic in $\mathbb{P}^2$ intersecting $C_2$ with multiplicity $4$ at $\sigma(p)$ and $2$ at $q_0$. 
The induced $\mathbb{A}^1$-fibration on $V_{0,2}$ has a unique degenerate fiber consisting of two reduced affine lines $H\cap V_{0,2}$ and $F_p\cap V_{0,2}$.  

b) The automorphism group of $(\mathbb{F}_1,D_1)$ acts on $D_1$ via the dihedral group of order $6$  generated by the symmetry with respect to the point $p_1=\sigma^{-1}([1:0:1])$ and the lift of the $\mathbb{Z}_3$-action on $C_1$ defined by $\varepsilon \cdot [x:y:z]=[\varepsilon x,\varepsilon^{-1}y:z]$. In particular, the induced action has no open orbit. 

For the pair $(\mathbb{F}_1,D_1)$, we have two types of pencils : the first family consists of the pencils $\mathcal{P}_{p_{\varepsilon^k}}$  at the points $p_{\varepsilon^k}=\sigma^{-1}([1:0:\varepsilon^k])$, $k=0,1,2$. 
These are generated respectively by $D_1$ and $L_{\varepsilon^k}+2F_{\varepsilon^k}$ where $L_{\varepsilon^k}\sim C_0+F$ is the strict transform of the tangent line to $C_1$ at the point $[1:0:\varepsilon^k]$ and $F_{\varepsilon^k}=\pi_1^{-1}(\pi_1(p_{\varepsilon^k}))$. 
The induced $\mathbb{A}^1$-fibrations on $V_{0,1}=\mathbb{F}_1\smallsetminus D_1$ have a unique degenerate fiber consisting of the disjoint union of two affine lines $L_{\varepsilon^k}\cap V_{0,1}$ and $F_{\varepsilon^k}\cap V_{0,1}$, the second occurring with multiplicity $2$. 

On the other hand, for every point $p\in D_1\smallsetminus \{p_1,p_{\varepsilon},p_{\varepsilon^2}\}$, the pencil $\mathcal{P}_p$ is generated by $D_1$ and $H_p+F_p$ where $F_p=\pi_1^{-1}(\pi_1(p))$, and $H_p$ is the strict transform of the unique smooth conic in $\mathbb{P}^2$ intersecting $C_1$ with multiplicity $4$ at $\sigma(p)$ and $2$ at $q_0$ if $p\in D_1\smallsetminus C_0$ or the strict transform of one of the two smooth conics intersecting $C_1$ with multiplicity $6$ at $q_0$ otherwise. 
In each case the induced $\mathbb{A}^1$-fibration on $V_{0,1}$ has unique degenerate fiber consisting of the disjoint union of two reduced affine lines $H\cap V_{0,1}$ and $F_p\cap V_{0,1}$.   

In contrast with the previous case, the description of the action of $\Aut(\mathbb{F}_1,D_1)$ on $D_1$ implies that even though the $\mathbb{A}^1$-fibrations on $\mathbb{F}_1\setminus D_1$ induced by the pencils $\mathcal{P}_p$, $p\in D_1\smallsetminus \{p_1, p_{\varepsilon},p_{\varepsilon^2}\}$ are abstractly isomorphic, they are no longer pairwise conjugate via elements of $\Aut(\mathbb{F}_1,D_1)$. \\

2) In the second case $(\mathbb{F}_3,B)$, a section $B$ of $\pi_3$ with self-intersection $5$ is linearly equivalent to  $C_0+4F$ where  $C_0$ is the exceptional section of $\pi_3$ with self-intersection $-3$ and $F$ is a fiber of $\pi_3$. 
Since the automorphism group of $\mathbb{F}_3$ acts transitively on such sections, there exists a unique model $(\mathbb{F}_3,D_3)$ up to isomorphism of pairs. 
Furthermore, the automorphism group of $(\mathbb{F}_3,D_3)$ acts on $D_3$ with two orbits: the point $p_\infty=D_3\cap C_0$ and its complement. 
The pencil $\mathcal{P}_{p_\infty}$ is generated by $D_3$ and $C_0+4F_\infty$ where $F_\infty=\pi_3^{-1}(\pi_3(p_\infty))$ and it restricts on $W_0=\mathbb{F}_3\smallsetminus D_3$ to an $\mathbb{A}^1$-fibration over $\mathbb{A}^1$ with a unique degenerate fiber consisting of two affine lines $C_0\cap W_0$ and $F_\infty\cap W_0$, the second one occurring with multiplicity $4$. For every other point $p\in D_3\smallsetminus{p_\infty}$, the rational pencil $\mathcal{P}_p$ is generated by $D_3$ and $H+F_p$ where $F_p=\pi_3^{-1}(\pi_3(p))$, and $H\sim C_0+3F$ is the unique section of $\pi_3$ intersecting $D_3$ at $p$ only with multiplicity $4$. The induced $\mathbb{A}^1$-fibration on $W_0$ has a unique degenerate fiber consisting of two reduced affine lines $H\cap W_0$ and $F_p\cap W_0$.  

\subsection{Automorphisms of the complement of a section with self-intersection $6$ in $\F_0$ }\label{expl:+6}

In this case a further new phenomenon occurs: the existence of uncountably many isomorphy types of smooth completions $(S,B_S)$, only finitely many of these having non-trivial automorphism groups. 
Below we only summarize these possible abstract isomorphy types and observe what is strictly necessary to finish the proof of Proposition  \ref{prop:Affautos}. 
The three types of possible models of smooth completions $(S,B_S)$ of an affine surface with $B_S^{2}=6$ are $(\mathbb{F}_0,C)$, $(\mathbb{F}_2,C)$ and $(\mathbb{F}_4,C)$ where $C$ is each time an ample section with self-intersection $6$. \\

1) The case $(\mathbb{F}_4,C)$: a section $C$ with $C^2=6$ is linearly equivalent to $C_0+5F$ where $C_0$ is the exceptional section of $\pi_4\colon\mathbb{F}_4\rightarrow \mathbb{P}^1$ with self-intersection $-4$ and $F$ is a fiber of $\pi_4$. 
Note that $C$ intersects $C_0$ transversally in a unique point $p_{\infty,4}$. The automorphism group of  $\mathbb{F}_4$ acts transitively on the set of such sections and, identifying $\mathbb{F}_4\smallsetminus (C_0\cup F_\infty)$ where $F_\infty=\pi_4^{-1}(\pi_4(p_{\infty,4}))$ with $\mathbb{A}^2$ in a similar way as in \S \ref{expl:+4}, we may assume that $C=C_4$ is the closure of the affine quintic $\{y=x^5\}\subset \mathbb{A}^2$. 
The automorphism group of $(\mathbb{F}_4,C_4)$ acts on $C_4$ with two orbits: the point $p_{\infty,4}=C_4\cap C_0$ and its complement. \\

2) The case $(\mathbb{F}_2,C)$: a section $C$ with $C^2=6$ is linearly equivalent to $C_0+4F$ where $C_0$ is the exceptional section of $\pi_2\colon\mathbb{F}_2\rightarrow \mathbb{P}^1$ with self-intersection $-2$ and $F$ is a fiber of $\pi_2$. 
Such a section intersects $C_0$ either in a single point with multiplicity two or transversally in two distinct points.

a) In the first case, up to an automorphism of $\mathbb{F}_2$ we may assume that $C=C_{2,1}$ is the closure in $\mathbb{F}_2$ of the intersection of the quartic $\{yz^3=x^4\}\subset \mathbb{P}^2$ with $\mathbb{A}^2$. 
The group $\Aut(\mathbb{F}_2,C_{2,1})$ acts on $C_{2,1}$ with three orbits: the point $p_{\infty,2}=C_{2,1}\cap C_0$, the point $p_{0,2}=(0,0)\subset\mathbb{A}^2\subset \mathbb{F}_2$ and their complement.

b) In the second case, up to an automorphism of $\mathbb{F}_2$ we may assume that $C=C_{2,2}$ is the closure in $\mathbb{F}_2$ of the intersection of the quartic $\{xyz^2=x^4-z^4\}\subset  \mathbb{P}^2$ with $\mathbb{A}^2$. 
The group $\Aut(\mathbb{F}_2,C_{2,2})$ acts on $C_{2,2}$ via the dihedral group of order $8$ generated by the symmetry with center at the point  
$p_s=[1:0:1]$ and the lift of the $\mathbb{Z}_4$-action on $\{xyz^2=x^4-z^4\}$ defined by $\varepsilon \cdot [x:y:z]=[\varepsilon x:\varepsilon^{-1} y : z]$.  \\

3) The case $(\mathbb{F}_0,C)$: a section $C$ of the first projection $\pi_0=\mathrm{pr}_1\colon\mathbb{F}_0=\mathbb{P}^1\times \mathbb{P}^1$ with $C^2=6$ is linearly equivalent to $C_0+3F$ where $C_0$ is a fiber of $\mathrm{pr}_2$ and $F$ a fiber of $\pi_0$. 
Such sections can be first roughly divided into three classes according to the number of fibers of the second projection which intersect $C$ with multiplicity $3$. 

a) If there exist at least two such fibers intersecting $C$ with multiplicity $3$ then the pair $(\mathbb{F}_0,C)$ is isomorphic to $(\mathbb{F}_0,C_{0,0})$ where $C_{0,0}=\{u_1^3v_0+u_0^3v_1=0\}$. 
The group  $\Aut(\mathbb{F}_0,C_{0,0})$ is then isomorphic to $\mathbb{C}^*\times \mathbb{Z}_2$ where $\mathbb{C}^*$ acts by $\lambda\cdot ([u_0:u_1],[v_0:v_1])=([\lambda u_0:u_1],[\lambda^3v_0:v_1])$ and where $\mathbb{Z}_2$ exchanges $u_0,v_0$ with $u_1,v_1$.

b) If there exists a unique fiber of $\mathrm{pr}_2$ intersecting $C$ with multiplicity $3$, then the pair $(\mathbb{F}_0,C)$ is isomorphic to $(\mathbb{F}_0,C_{0,1})$ where $C_{0,1}=\{u_1^3v_0+u_0^2(u_0+u_1)v_1=0\}$. 
Its automorphism group is isomorphic to $\mathbb{Z}_2$, acting via $([u_0:u_1],[v_0:v_1])\mapsto ([-u_0-2u_1/3:u_1],[-v_0-4v_1/27:v_1])$.

c) Finally, if there is no fiber of $\mathrm{pr}_2$ intersecting $C$ with multiplicity $3$ then the pair $(\mathbb{F}_0,C)$ is isomorphic to a pair of the form $(\mathbb{F}_0,C_{1,b})$ where $C_{1,b}=\{u_1^2(u_0+u_1)v_0+u_0^2(u_0+bu_1)v_1=0\}$ for some $b\in\mathbb{C}\smallsetminus\{0,1\}$ such that the polynomial $s(t)=2t^2+(b+3)t+2b$ has simple roots (this last condition guarantees precisely that $C_{1,b}$ cannot intersect a fiber of $\mathrm{pr}_2$ with multiplicity $3$). 
Furthermore, such a curve $C_{1,b}$ has exactly four horizontal tangents at the following points $P_i(b)=(p_i(b),q_i(b))$: $P_1(b)=([0:1],[0:1])$, $P_2(b)=([1:0],[1:0])$, $P_3(b)=([r_1:1],[r_1^2(r_1+b)/(r_1+1):1])$ and $P_4(b)=([r_2:1],[r_2^2(r_2+b)/(r_2+1):1])$, where $r_1,r_2\in \mathbb{C}\smallsetminus\{-1\}$ are the roots of $s(t)$. 
It follows from this description that two pairs $(\mathbb{F}_0,C_{1,b})$ and $(\mathbb{F}_0,C_{1,b'})$ are isomorphic only if there exists a permutation $\sigma\in\mathfrak{S}_4$ such that the cross-ratios of $(p_1(b),p_2(b),p_3(b),p_4(b))$ (resp. $(q_1(b),q_2(b),q_3(b),q_4(b))$) and $(p_{\sigma(1)}(b'),p_{\sigma(2)}(b'),p_{\sigma(3)}(b'),p_{\sigma(4)}(b'))$ (resp. $(q_{\sigma(1)}(b'),q_{\sigma(2)}(b'),q_{\sigma(3)}(b'),q_{\sigma(4)}(b'))$) are equal. 
A direct computation implies in turn that there exists uncountably many isomorphy classes of such pairs all having a finite group of automorphism of order at most $24$ and that this group is in fact trivial except for finitely many of these.\\

\bibliographystyle{amsplain}
\bibliography{biblio_surface}

\end{document}

%% file: graph_surface.tex
\newcommand{\mygraph}[1]{\xybox{\xygraph{#1}}}
\newcommand{\size}{0.56cm}
\newcommand{\sizemodel}{0.38cm}
\newcommand{\gros}[1]{\mbox{\bf \LARGE \it #1}}

\newcommand{\ExpleRefS}{
\mygraph{ !{<0cm,0cm>;<\size,0cm>:<0cm,\size>::}
!{(0,0)}="A"
!{(2,0)}="B"
!{(0,1)}*{S = \cpd}
"A"-_{E_0}^(0.6){+1}"B"} 
}

\newcommand{\ExpleRefSprime}{
\mygraph{ !{<0cm,0cm>;<\size,0cm>:<0cm,\size>::}
!{(0,0)}="A"
!{(2,0)}="B"
!{(0,1)}*{S'}
"A"-_{E'_0}|<(0.8){\star}"B"} 
}

\newcommand{\ExpleRefSdouble}{
\mygraph{ !{<0cm,0cm>;<\size,0cm>:<0cm,\size>::}
!{(0,0)}="A"
!{(2,0)}="B"
!{(0,1)}*{S''}
"A"-_{E''_0}|<(0.2){\star}|<(0.8){\star}"B"} 
}

\newcommand{\ExpleRefX}{
\mygraph{ !{<0cm,0cm>;<\size,0cm>:<0cm,\size>::}
!{(0,0)}="A"
!{(2,0.5)}="B"
!{(1.5,0.5)}="C"
!{(3.5,0)}="D"
!{(3,0)}="E"
!{(5,0.5)}="F"
!{(4.5,0.5)}="G"
!{(6.5,0)}="H"
!{(1,2)}*{X = \tilde S''}
"A"-^{E_0}_{-2}"B" "C"-^{E_0''}_{-1}"D"
"E"-^{E_0'}_{-2}"F" "G"-^{E}_{-2}"H" } 
}

\newcommand{\ExpleRefSprimetilde}{
\mygraph{ !{<0cm,0cm>;<\size,0cm>:<0cm,\size>::}
!{(0,0)}="A"
!{(2,0.5)}="B"
!{(1.5,0.5)}="C"
!{(3.5,0)}="D"
!{(0,1)}*{\tilde S'}
"A"-^{E_0'}_0"B" "C"-^{E}_{-2}"D"} 
}

\newcommand{\dessinresolutionbasic}{
\mygraph{
!{<0cm,0cm>;<\size,0cm>:<0cm,\size>::}
!{(0,-2)  }="A" !{(4,-1)  }="B"
!{(3,-1)  }="C" !{(7,-2)  }="D"
!{(6,-2)  }="E" !{(10,-1) }="F"
!{(4.5,-2)}="G" !{(5.5,1)}="H"
!{(5.5,0)}="I" !{(4.5,3)}="J"
"A"-^{C_0}_{-1}"B" "C"-^(0.6){C_3}_(0.6){-2}"D" "E"-^{C_4}_{-1}"F" 
"G"-^{C_2}_{-2}"H" "I"-^{C_1}_{-2}"J"
}}

\newcommand{\dessinpreuve}{
\mygraph{ !{<0cm,0cm>;<\size,0cm>:<0cm,\size>::}
!{(0,0)  }="A" !{(4,1)  }="B"
!{(3,1)  }="C" !{(7,0)  }="D"
!{(1,0)  }="E" !{(0,3)  }="F"
!{(6,0)  }="G" !{(7,3)  }="H"
!{(7,2)  }="I" !{(6,5)  }="J"
!{(3,3)}*{\gros{Y}}
"A"-_{E_0}"B" "C"-_{E_1}"D"
"E"-@{~}^{\rm Sing}"F" "G"-@{~}^{\rm Aux}"H" "I"-@{~}^{\rm Sing}"J" 
}}

\newcommand{\dessinTa}{
\mygraph{ !{<0cm,0cm>;<\size,0cm>:<0cm,\size>::}
!{(0,-2)  }="A" !{(4,-1)  }="B"
!{(3,-1)  }="C" !{(7,-2)  }="D"
!{(1,-2)  }="E" !{(0,1)  }="F"
!{(6,-2)  }="G" !{(7,1)  }="H"
!{(7,0)  }="I" !{(6,3)  }="J"
"A"-_{E_i}^{-1}"B" "C"-_{E_{i+1}}^{-1}"D"
"G"-@{~}_{\rm Sing}^>(.7){-k,-2,\cdots,-2}"H"  
}}

\newcommand{\dessinTb}{
\mygraph{ !{<0cm,0cm>;<\size,0cm>:<0cm,\size>::}
!{(0,-2)  }="A" !{(4,-1)  }="B"
!{(3,-1)  }="C" !{(7,-2)  }="D"
!{(1,-2)  }="E" !{(0,1)  }="F"
!{(6,-2)  }="G" !{(7,1)  }="H"
!{(7,0)  }="I" !{(6,3)  }="J"
"A"-_{E_i}^{-1}"B" "C"-_{E_{i+1}}^{-1}"D"
"E"-@{~}^{\rm Sing}_>(.7){-2,\cdots,-2,-(k-1)}"F"  
}}

\newcommand{\dessinModeleI}{
\mygraph{
!{<0cm,0cm>;<\size,0cm>:<0cm,\sizemodel>::}
!{(2,-2) }="A1" !{(1,1)  }="A2" 
!{(1,0)  }="B1" !{(2,3)  }="B2" 
!{(2,2)  }="C1" !{(1,5)  }="C2"
!{(1,4)  }="D1" !{(2,7)  }="D2"
!{(1,6.5)  }="E1" !{(6,6.5)  }="E2"
!{(5,7)  }="F1" !{(5,3)  }="F2"
!{(1,5.5)  }="S1" !{(4,4)  }="S2"
!{(4.5,2)}*{\mbox{Model I}}
"A1"-@{.}^{E_3}_{-1}"A2" "B1"-^{E_2}_{-2}"B2" "C1"-^{E_1}_{-2}"C2" "D1"-^>(.6)l_>(.6){-2}"D2" 
"E1"-^>(.4){C}^>(.6){-1}"E2" "F1"-^{0}^>(.8)B"F2"
"S1"-@{.}^>(0.8){E_4}_>(0.8){-1}"S2"
}}

\newcommand{\dessinModeleII}{
\mygraph{
!{<0cm,0cm>;<\size,0cm>:<0cm,\sizemodel>::}
!{(2,-2) }="A1" !{(1,1)  }="A2" 
!{(1,0)  }="B1" !{(2,3)  }="B2" 
!{(2,2)  }="C1" !{(1,5)  }="C2"
!{(1,4)  }="D1" !{(2,7)  }="D2"
!{(1,6.5)  }="E1" !{(6,6.5)  }="E2"
!{(5,7)  }="F1" !{(5,3)  }="F2"
!{(1,1.5)  }="S1" !{(4,1)  }="S2"
!{(5.5,2)}*{\mbox{Model III}}
"A1"-@{.}^{E_1}_{-3}"A2" "B1"-^>(.55){E_3}_>(.3){-2}"B2" "C1"-^{E_2}_{-2}"C2" "D1"-^>(.6)l_>(.6){-2}"D2" 
"E1"-^>(.4){C}^>(.6){-1}"E2" "F1"-^{0}^>(.8)B"F2"
"S1"-@{.}^>(0.8){E_4}_>(0.8){-1}"S2"
}}

\newcommand{\dessinModeleIII}{
\mygraph{
!{<0cm,0cm>;<\size,0cm>:<0cm,\sizemodel>::}
!{(2,-2) }="A1" !{(1,1)  }="A2" 
!{(1,0)  }="B1" !{(2,3)  }="B2" 
!{(2,2)  }="C1" !{(1,5)  }="C2"
!{(1,4)  }="D1" !{(2,7)  }="D2"
!{(1,6.5)  }="E1" !{(6,6.5)  }="E2"
!{(5,7)  }="F1" !{(5,3)  }="F2"
!{(1,3)  }="S1" !{(4,4)  }="S2"
!{(4.5,2)}*{\mbox{Model II}}
"A1"-@{.}^{E_4}_{-1}"A2" "B1"-^{E_3}_{-2}"B2" "C1"-^>(.6){E_2}_>(.6){-2}"C2" "D1"-^l_{-2}"D2" 
"E1"-^>(.4){C}^>(.6){-1}"E2" "F1"-^{0}^>(.8)B"F2"
"S1"-@{.}^>(0.8){E_1}_>(0.8){-2}"S2"
}}

\newcommand{\dessinFun}{
\mygraph{
!{<0cm,0cm>;<0.4cm,0cm>:<0cm,0.4cm>::}
!{(2,4)  }="D1" !{(2,1)  }="D2"
!{(1,3.5)}="E1" !{(6,3.5)  }="E2"
!{(5,4)  }="F1" !{(5,1)  }="F2"
!{(4,1)}*{\F_1}
"D1"-_l^{0}"D2" "E1"-^C_{-1}"E2" "F1"-_0^{B}"F2"
}}

\newcommand{\dessinFdeuxbis}{
\mygraph{
!{<0cm,0cm>;<0.4cm,0cm>:<0cm,0.4cm>::}
!{(1,3.5)}="E1" !{(6,3.5)  }="E2"
!{(2,4)  }="F1" !{(2,1)  }="F2"
!{(4,1)}*{\F_2}
"E1"-@{.}_{-2}^{E_1}"E2" "F1"-@{.}^{0}_{E_4}"F2"
}}

\newcommand{\dessinFzero}{
\mygraph{
!{<0cm,0cm>;<0.4cm,0cm>:<0cm,0.4cm>::}
!{(1,3.5)}="E1" !{(6,3.5)  }="E2"
!{(2,4)  }="F1" !{(2,1)  }="F2"
!{(4,1)}*{\F_0}
"E1"-@{.}_{0}^{E_1}"E2" "F1"-@{.}^{0}_{E_4}"F2"
}}

\newcommand{\dessinFzerobis}{
\mygraph{
!{<0cm,0cm>;<0.4cm,0cm>:<0cm,0.4cm>::}
!{(1,3.5)}="E1" !{(6,3.5)  }="E2"
!{(2,4)  }="F1" !{(2,1)  }="F2"
!{(4.5,1)}*{\F_0 \mbox{ or } \F_2}
"F1"-@{.}^{0}_{E_3}"F2" "E1"-@{.}^>(.6){0 \mbox{ or }-2}"E2"
}}

\newcommand{\ModeleI}{
\mygraph{
!{<0cm,0cm>;<\size,0cm>:<0cm,\size>::}
!{(0,0)  }*+{\dessinFzerobis}="F0"
!{(0,6)  }*{\dessinModeleI}="MI"
"MI"-@{->}"F0" 
}}

\newcommand{\ModeleII}{
\mygraph{
!{<0cm,0cm>;<\size,0cm>:<0cm,\size>::}
!{(0,0)  }*+{\dessinFdeuxbis}="F2"
!{(0,6)  }*{\dessinModeleII}="MII"
"MII"-@{->}"F2" 
}}

\newcommand{\ModeleIII}{
\mygraph{
!{<0cm,0cm>;<\size,0cm>:<0cm,\size>::}
!{(0,0)  }*+{\dessinFzero}="F0"
!{(0,6)  }*{\dessinModeleIII}="MIII"
"MIII"-@{->}"F0" 
}}

\newcommand{\dessinTrFzero}{
\mygraph{
!{<0cm,0cm>;<0.4cm,0cm>:<0cm,0.4cm>::}
!{(1,3.5)}="E1" !{(6,3.5)  }="E2"
!{(2,4)  }="F1" !{(2,1)  }="F2"
!{(4.5,1)}*{\F_0}
"F1"-^{0}_{E_3}"F2" "E1"-_0^{\Delta}"E2"
}}

\newcommand{\dessinTrFun}{
\mygraph{
!{<0cm,0cm>;<0.4cm,0cm>:<0cm,0.4cm>::}
!{(.5,3.5)}="H1" !{(6,3.5)  }="H2"
!{(1,4)  }="G1" !{(1,.5)  }="G2"
!{(.5,1)}="B1" !{(6,1)  }="B2"
!{(5.5,4)  }="D1" !{(5.5,.5)  }="D2"
!{(-1.5,1.85)  }="delD" !{(3.25,1.75)  }="delM" !{(6.1,4)  }="delF"
!{(-1,3)}*{\F_1}
"G1"-_>(.4){l}_>(.6){0}"G2" "H1"-^>(.6){-1}^>(.4){C}"H2"
"B1"-_>(.6){0}_>(.4){L}"B2" "D1"-^>(.4){B}^>(.6){0}"D2"
"delD"-@/_.32cm/"delM" "delM"-^<{\Delta}_>(.3){+3}@/_.05cm/"delF"
}}

\newcommand{\dessinTrFdeux}{
\mygraph{
!{<0cm,0cm>;<0.4cm,0cm>:<0cm,0.4cm>::}
!{(1,3.5)}="E1" !{(6,3.5)  }="E2"
!{(2,4)  }="F1" !{(2,1)  }="F2"
!{(4.5,1)}*{\F_2}
"F1"-^{0}_{E_3}"F2" "E1"-^{C_0= L}_{-2}"E2"
}}

\newcommand{\dessinTrProj}{
\mygraph{
!{<0cm,0cm>;<0.4cm,0cm>:<0cm,0.4cm>::}
!{(2.5,3.5)  }="G1" !{(1,.5)  }="G2"
!{(-.3,0.9)}="B1" !{(5,1.1)  }="B2"
!{(0,4.3)  }="D1" !{(3.8,.8)  }="D2"
!{(-1,1.2)  }="delD" !{(0,3.9)  }="delF"
!{(2,2.4)}*{\bullet}
!{(-1.5,3)}*{\mathbb{P}^2}
"G1"-_>(.6){l}"G2" 
"B1"-_>(.9){L}"B2" "D1"-^>(.1){B}^>(.6){q}"D2"
"delD"-@/_1.3cm/^>(.1){+4}^<{\Delta}"delF" 
}}

\newcommand{\dessinTrHaut}{
\mygraph{
!{<0cm,0cm>;<\size,0cm>:<0cm,\sizemodel>::}
!{(2,-2) }="A1" !{(1,1.5)  }="A2" 
!{(1,0.5)  }="B1" !{(2,4)  }="B2" 
!{(2,3)  }="C1" !{(1,6.5)  }="C2"
!{(1,5.5)  }="D1" !{(2,9)  }="D2"
!{(1,8.5)  }="E1" !{(4.5,8.5)  }="E2"
!{(5.5,6.5)  }="F1" !{(5.5,1)  }="F2"
!{(4.8,6.7)  }="B'H" !{(4,1.9)  }="B'B"
!{(1,7.5)  }="S1" !{(3.3,6)  }="S2"
!{(4,9)  }="ExH" !{(6,6)  }="ExB"
!{(1,2.4)  }="delD" !{(4.4,2.7)  }="delF" 
!{(1,-2)  }="BD" !{(6,2)  }="BF" 
"A1"-^{E_3}_>(.6){-1}"A2" "B1"-^>(.35){E_2}^>(.65){-2}"B2" "C1"-^{E_1}_{-2}"C2" "D1"-^>(.6)l_>(.6){-2}"D2" 
"B'H"-^>(.6){B'}_>(.6){-1}"B'B" 
"E1"-_{-1}^>(.6)C"E2" "F1"-^{-1}^>(.65)B"F2"
"S1"-^>(0.8){E_4}_>(0.95){-1}"S2" 
"ExH"-^>(.6){-2}@/_0.5cm/"ExB" 
"delD"-_>(.35){\Delta}^>(.45){-1}@/_0.3cm/"delF"
"BD"-@/^0.2cm/_>(.5){L}_>(.65){-2}"BF"
}}

\newcommand{\dessinTrDeb}{
\mygraph{
!{<0cm,0cm>;<\size,0cm>:<0cm,\sizemodel>::}
!{(2,-2) }="A1" !{(1,1)  }="A2" 
!{(1,0)  }="B1" !{(2,3)  }="B2" 
!{(2,2)  }="C1" !{(1,5)  }="C2"
!{(1,4)  }="D1" !{(2,7)  }="D2"
!{(1,6.5)  }="E1" !{(6,6.5)  }="E2"
!{(5,7)  }="F1" !{(5,1)  }="F2"
!{(1,5.5)  }="S1" !{(3.3,4)  }="S2"
!{(1,2)  }="delD" !{(3,2)  }="delM" !{(5.3,7)  }="delF" 
!{(1,-2)  }="BD" !{(6,2)  }="BF" 
!{(2,8)}*{\hat{S}\mbox{ (model I)}}
"A1"-^{E_3}_>(.6){-1}"A2" "B1"-^>(.35){E_2}^>(.65){-2}"B2" "C1"-^{E_1}_{-2}"C2" "D1"-^>(.6)l_>(.6){-2}"D2" 
"E1"-_{-1}^>(.6)C"E2" "F1"-^{0}^>(.8)B"F2"
"S1"-^>(0.8){E_4}_>(0.95){-1}"S2" 
"delD"-@/_.5cm/"delM" "delM"-_>(.35){\Delta}_>(.2){+1}@/^0.1cm/"delF"
"BD"-@/^0.2cm/_>(.5){L}_>(.65){-2}"BF"
}}

\newcommand{\dessinTrFin}{
\mygraph{
!{<0cm,0cm>;<\size,0cm>:<0cm,\sizemodel>::}
!{(2,-2) }="A1" !{(1,1)  }="A2" 
!{(1,0)  }="B1" !{(2,3)  }="B2" 
!{(2,2)  }="C1" !{(1,5)  }="C2"
!{(1,4)  }="D1" !{(2,7)  }="D2"
!{(1,6.5)  }="E1" !{(6,6.5)  }="E2"
!{(5,7)  }="F1" !{(5,1)  }="F2"
!{(1,5.5)  }="S1" !{(3.3,4)  }="S2"
!{(1,2)  }="delD" !{(5.5,2.5)  }="delF" 
!{(1,-2)  }="BD" !{(4,0)  }="BM1" 
!{(4.2,6.9)  }="BF" !{(5.5,6)  }="BM2" 
!{(4.5,8)}*{\hat{S}'\mbox{ (model I)}}
"A1"-^{E_3}_>(.6){-1}"A2" "B1"-^>(.35){E_2}^>(.65){-2}"B2" "C1"-^{E_1}_{-2}"C2" "D1"-^>(.6)l_>(.6){-2}"D2" 
"E1"-_>(.6){-1}^C"E2" "F1"-^>(.55){0}^>(.4){B'}"F2"
"S1"-^>(0.8){E_4}_>(0.95){-1}"S2" 
"delD"-@/_.4cm/_>(.45){\Delta}^>(.55){-1}"delF"
"BD"-@/^0.2cm/"BM1" "BM1"-@/_1.3cm/_>(.5){0}_>(.6){L}"BM2" 
"BM2"-@/_0.05cm/"BF"
}}